\documentclass[reqno]{amsart} 

\usepackage{
  bm,
  bbm,
  color,
  a4wide,
  amssymb,
  amsmath,
  latexsym,
  calligra,
  mathrsfs,
  stmaryrd,
  hyperref,
  graphicx,
  mathtools,
  enumitem,
  tikz,
  booktabs,
  cancel
}
\usepackage[normalem]{ulem}
\usepackage[2cell,arrow,curve,matrix]{xy}

\newcommand{\p}{\mathfrak{p}}
\newcommand{\q}{\mathfrak{q}}

\newcommand{\fa}{\mathfrak{a}}

\newcommand{\fm}{\mathfrak{m}}
\newcommand{\fr}{\mathfrak{r}}

\newcommand{\fF}{\mathfrak{F}}

\newcommand{\fP}{\mathfrak{P}}

\newcommand{\bC}{\mathbf{C}}
\newcommand{\bF}{\mathbf{F}}

\newcommand{\SL}   {\mathbf{SL}}
\newcommand{\CM}   {\mathbf{CM}}
\newcommand{\Mon}  {\mathbf{Mon}}
\newcommand{\FGMon}{\mathbf{FGMon}}

\newcommand{\cA}{\mathcal A}
\newcommand{\cB}{\mathcal B}
\newcommand{\cC}{\mathcal C}
\newcommand{\cG}{\mathcal G}
\newcommand{\cE}{\mathcal E}
\newcommand{\cF}{\mathcal F}
\newcommand{\cH}{\mathcal H}

\newcommand{\cO}{\mathcal O}

\newcommand{\F}{\mathbb F}
\newcommand{\I}{\mathbb I}
\newcommand{\N}{\mathbb N}

\newcommand{\bH}{\mathbb H}

\newcommand{\sC}{\mathsf{C}}
\newcommand{\sD}{\mathsf{D}}

\DeclareMathOperator*{\Spec} {\mathsf{Spec}}

\DeclareMathOperator {\Hom}  {\mathsf{Hom}}

\DeclareMathOperator {\id}   {\mathsf{Id}}

\DeclareMathOperator*{\Aut}  {\mathsf{Aut}}
\DeclareMathOperator*{\Ab}   {\mathsf{Ab}}
\DeclareMathOperator*{\colim}{\mathsf{colim}}

\DeclareMathOperator*{\tl}   {2\text{-}\mathsf{lim}}

\DeclareMathOperator*{\To}   {\,\Rightarrow\,}
\DeclareMathOperator*{\coex} {\mathsf{Coext}}
\DeclareMathOperator*{\im}   {\mathop{\sf Im}\nolimits}

\DeclareMathOperator*{\CoKer}  {\mathop{\mathsf{CoKer}}\nolimits}

\renewcommand{\ker}{\mathsf{Ker}}

\newcommand{\xto}  [1]   {\xrightarrow{#1}}
\newcommand{\xfrom}[1]   {\xleftarrow{#1}}
\newcommand{\Msl}  [1][M]{#1^\mathsf{sl}}
\renewcommand{\sl}       {(-)^\mathsf{sl}}

\newcommand{\qtext}[1]{\quad\text{#1}\quad}

\newcommand{\pt}{\bullet}

\let\scong\cong
\renewcommand{\cong}{\;\scong\;}

\newcommand{\capcup}{%
    \mathbin{
      \ooalign{
        \(\displaystyle\cap\)\cr
        \hidewidth
        \hspace{0.7ex}%
        \raisebox{0.1ex}{\scalebox{0.5}{\(\cup\)}}
        \hidewidth\cr
    }
  }
}

\newtheorem{Core}{Definition}[section]
\newtheorem{Aux}{Remark}[section]

\newtheorem{De}[Core]{Definition}
\newtheorem{Th}[Core]{Theorem}
\newtheorem{Pro}[Core]{Proposition}
\newtheorem{Le}[Core]{Lemma}
\newtheorem{Cor}[Core]{Corollary}

\newtheorem{Rem}[Aux]{Remark}

\newtheorem{Const}[Aux]{Construction}
\newtheorem{Conv}[Aux]{Convention}
\newtheorem{Nota}[Aux]{Notation}

\author{Ilia Pirashvili}
\address{University of Galway, Áras De Brún, Gaillimh/Galway, H91 H3CY, Ireland}
\email{ilia\_p@ymail.com (personal)}
\email{ilia.pirashvili@universityofgalway.ie (uni)}


\begin{document}
\title[Deformations of Monoid Schemes]{Deformation Theory of Monoid Schemes I}

\begin{abstract}
  The aim of this paper is to develop a deformation theory of monoid schemes,
  generalising the approach developed by Grillet \cite{gr}. The core idea of this
  approach is to introduce the notion of a system of abelian groups, as the naive
  approach to exactness does not work for monoids.

  We first study the case of monoid sheaves (functors over a poset into the category of
  monoids) and prove a classification theorem in this setting, showing that the
  coextensions \(\coex(\cF, \cA)\) of a monoid functor \(\cF\) with a system of abelian
  groups \(\cA\) is a symmetric categorical group and equivalent to the one obtained by
  the abelian group homomorphism \([\cC^0 \to \ker\partial^1]\), thereby linking with
  cohomology of certain types of complexes, as expected.

  We then move towards monoid schemes, which are a type of a monoid sheaf, but where
  localisations now allow us to develop our most noteworthy result: We show that
  coextensions can be seen in a natural way as a stack of symmetric categorical groups.
  We will mention a few mild implications of this, but leave the deeper uses of stack
  theory in this setting for later papers.
\end{abstract}

\maketitle

\section*{Introduction}

  Monoids are a versatile topic with many different viewpoints. They can be viewed from
  many angles, including as generalisations of groups, as one-object categories and as
  an analogue of commutative rings. The latter becoming especially prominent over the
  last 20 years under the name of \(\F_1\)-geometry, see for instance
  \cite{deitmar1,deitmar2,cc1,cc2,chu,cortinas}. It is also the one that perhaps needs
  most explaining to these unfamiliar with it: The idea is to regard monoids are rings
  without addition and proceed to do algebraic geometry as one would do over
  commutative rings. Using prime ideals and localisations, we can define gluing of
  monoids exactly like gluing of rings and thereby create objects we will refer to as
  monoid schemes.

  This paper has aspects from all of the above aspects of monoids, but its motivation
  comes mainly from the algebraic geometric view. The long-term goal is to develop a
  deformation theory of monoids and monoid schemes in the vein of
  Quillen~\cite{quillen} and Illusie~\cite{ill}. This will be addressed in our upcoming
  papers. The current one only deals with the low dimensional case.

  Much of our work in the affine case is known, but this paper offers perhaps a more
  modern description. The results for monoid schemes are new and at times, offer
  perhaps even a few surprises.

  \medskip

  In classical (ring-theoretic) algebraic geometry, deformation theory is one of the
  major tools used for its study. It controls the infinitesimal deformations of a
  scheme and gives a systematic way of studying ``closely related'' families of
  schemes. Two of its major architects are Quillen~\cite{quillen} in the affine case
  and Illusie~\cite{ill} in the general scheme setting.

  \medskip

  Our aim is to develop a natural analogue of this theory for monoid schemes. The
  affine analogue of a ring extension in the monoidal world is a \emph{group
  coextension}, a notion notion developed by Grillet~\cite{gr} in his study of the
  cohomology of commutative semigroups. A core issue of monoid theory is that kernels
  in the naive sense are not enough to capture monoid homomorphisms in a sufficient
  manner. There are numerous approaches to solve such problems, and one of the most
  interesting ones is to consider a system of abelian groups. This is constructed by
  regarding a monoid \(M\) as a category, whose objects are the elements of \(M\) and
  whose morphisms correspond to multiplication, denoted \(\bH(C)\) and consider
  functors from \(\bH\) to abelian groups. This is called a system of abelian groups
  over \(M\).

  Using this, we can define a coextension of a commutative monoid \(M\) by a system of
  abelian groups, in accordance to Grillet, as a surjective monoid homomorphism \(\pi:
  N \to M\) together with regular, coherent and compatible actions of the abelian
  groups \(\cA(m)\), for each fibre \(\pi^{-1}(m)\). Systems of abelian groups are
  equivalent to group objects in the slice category over a monoid \(M\). We also very
  quickly verify in Proposition~\ref{prop:monoig_coex_generalises_group_coex}, for the
  convenience of the reader, that this notion of coextension is indeed a generalisation
  of the classical coextension of groups.

  \medskip

  Analogous to the classical setting, Grillet proved that the isomorphism classes of
  coextensions of \(M\) by \(\cA\) are classified by the first cohomology group
  \(\sD^1(M, \cA)\) of a cochain complex
  \[
    \sC^0(M, \cA) \xto{\;\partial^0\;} \sC^1(M, \cA)
    \xto{\;\partial^1\;} \sC^2(M, \cA).
  \]
  We will call this the \emph{Grillet complex}. The group \(\sD^0(M, \cA) = \ker
  \partial^0\) is the group of derivations from \(M\) with values in \(\cA\).

  \medskip

  The main aim of this paper is to generalise this theory in a categorically rigorous
  manner to the general (non-affine) setting. In doing so, we will actually do two
  further interrelated things: The first will be to give a self-contained and
  categorically more modern description of the affine setting. This includes the Baer
  sum, the versal coextension, and the long exact sequence of Grillet cohomology.

  The second is to develop a theory of coextensions of monoid sheaves/functors, which
  is more general than the theory of coextensions of monoid schemes. However, when
  specialising to monoid schemes, we will show that localisations enable us to define
  additional structures on the coextensions. In particular, arguably our main result,
  will be to show that the coextensions, \(\coex(X, \cA)\), can be regarded in a
  natural way as a stack of symmetric categorical groups, whose global section exactly
  describes the coextensions of a monoid scheme \(X\) by a system of abelian groups
  \(\cA\).

  \medskip

  In more detail, our main results can be described as follows:

  \medskip

  Let \(M\) be a commutative monoid and \(\cA \in \cH(M)\) a system of abelian groups
  over \(M\). We show that the category \(\coex(M, \cA)\) of coextensions of \(M\) by
  \(\cA\) is a symmetric categorical group in a natural way: The monoidal structure
  (``addition'') is played by the Baer sum and the unit is the semi-direct product (see
  Section~\ref{sec:semidirect} for the definition of the semi-direct product with a
  system of abelian groups). We also show that this is equivalent to the symmetric
  categorical group associated to the group homomorphism \(C^0 \to \ker\partial^0\).
  That is to say, we have an equivalence
  \[
    \coex(M, \cA) \;\simeq\; \bigl[\sC^0(M, \cA) \xrightarrow{\;\partial^0\;}
    \ker\partial^1(M, \cA)\bigr].
  \]
  We refer to this as the (affine) classification theorem. The core mathematical
  argument of this can already be found implicitly in Grillet~\cite{gr}.

  In particular, we have
  \begin{eqnarray*}
    \pi_0(\coex(M,\cA)) & \cong & \sD^1(M, \cA) \\
    \pi_1(\coex(M,\cA)) & \cong & \sD^0(M, \cA).
  \end{eqnarray*}

  \medskip

  Our work in the affine case also includes developing the theory of versal
  coextensions. These are coextensions from which we have a morphisms of coextensions
  (meaning a tripple) to any other coextension of the same monoid by the same system of
  abelian groups.

  There are many reasons why these are important, but one such is that it allows us to
  write a certain type of exact sequence: Given a versal coextension \(0 \to \cB \to K
  \to M \to 0 \in \coex(M, \cA)\), we prove the existence of an exact sequence
  \[
    0 \to \sD^0(M,\cA) \to D(q) \to \Hom_{\cH(M)}(\cB, \cA)
    \to \sD^1(M,\cA) \to 0.
  \]
  The details can be found in Section~\ref{sec:versal}, but as a quick note we will say
  that \(D(q)\) is a group of derivations from \(K\) with values in \(\cA\) satisfying
  some compatibility condition. The main point is that it enables us to express
  \(\sD^1(M, \cA)\) in terms of simpler (more understandable) groups.

  We also provide a long exact sequence in the Grillet cohomology associated to a short
  exact sequence \(0 \to \cA' \to \cA \to \cA'' \to 0\), where \(\cA, \cA', \cA''\) are
  systems of abelian groups over the same monoid \(M\) (meaning, over \(\bH(M)\).

  \medskip

  After this, we move to the genuinely new setting, where the role of \(M\) will be
  played by functors \(\cF: P \to \Mon\) from a poset to the category of monoids. The
  reason why we restrict ourselves to poset rather than arbitrary categories is that we
  want to study the extension theory of monoid schemes. It is well known
  \cite{cortinas,p1,p2} that the underlying topological space of a monoid scheme is a
  poset topology, meanig a topology arising form a poset via ordering. Moreover,
  sheaves over such topologies are equivalent to mere functors over said posets. In
  particular, so are monoid schemes. Of course, the restriction functors must be
  localisations in that case, but no coherence condition is required when working over
  the underlying poset.

  As such, we will consider functors of a monoid scheme first, and then specialise to
  monoid schemes. We define what systems of abelian groups over a monoid functor
  \(\cF\) are and what a coextension is. These are natural generalisations of the
  affine case and, in particular, we have \(\cH(M) \simeq \cH(\Spec(M))\). We show that
  the structural and classification theorems of the affine case generalise fully for
  monoid functors. In particular, we show that \(\coex(\cF, \cA)\) is a symmetric
  categorical group and equivalent to the symmetric categorical group associated to the
  group homomorphism \([\delta: C_P^{0,0}(\cF, \cA) \to Z^1(\cF, \cA)]\). Here,
  \(C_P^{0,0}(\cF, \cA)\) and \(Z^1(\cF, \cA)\) are obtained from a truncated bicomplex
  \(C^{\bullet, *}_P(\cF, \cA)\) which we introduce and which generalises the Grillet
  complex in the affine case. We call this the global Grillet complex.

  \medskip

  We then specialise to the monoid scheme setting, where the additional structure of
  localisations, which we extend to localisations of systems, allows us to construct
  functors between coextensions. This is perhaps a little unexpected when coming from
  ring theory, but it does have very pleasant consequences. In particular, we are able
  to prove Theorem~\ref{thm:global-sections-2limit}, which states that that
  \[ x\mapsto \coex(\cO_{X,x}, \cA_x) \]
  is a stack of symmetric categorical groups. Recall that being a stack implies the
  equivalence of the following categories:
  \[
    \coex(X, \cA) \;\simeq\;
    2 \,\text{-}\, \lim_{x \in X}\, \coex(\cO_{X,x}, \cA_x).
  \]
  We will also consider the special case of semilattices in the last section.

  \bigskip

  The paper is organised as follows: Section~\ref{sec:spec} establishes the necessary
  preliminaries on monoid theory, including a quick summary of the key results on the
  spectrum, such as the reduction to semilattices from \cite{p1} and that sheaves over
  the Zariski topology is a functor over its underlying semilattice.

  Section~\ref{sec:hm} introduces the notion of systems of abelian groups over a
  monoid, denoted \(\cH(M)\), and deals with some of its basic properties, such as
  functoriality, localisation, and the semidirect product. We also show that they are
  equivalent to the category of abelian group objects in the slice category \(\CM/M\)
  over \(M\).

  We proceed with the theory of group coextension in Section~\ref{sec:coext}, where we
  define the category of coextensions, endow it with the Bear sum to make it into a
  symmetric categorical group, and define the notion of the localisation of
  coextensions. We also construct pullbacks and pushforwards and prove some important
  lemmas, such as the fact that monoids and their coextensions have isomorphic spectra,
  and the short five lemma.

  In Section~\ref{sec:grillet}, we move on towards the more systematic study of
  coextensions. We establishes the Grillet complex and prove the so called
  classification theorem. This theorem says that the symmetric categorical group of
  coextensions can be studied with the cohomology of the Grillet complex. We also
  define the versal extension and talk about its induced exact sequence, as well as the
  long exact sequence induced by short exact sequence of systems of abelian groups.

  We move towards the non-affine setting in Section~\ref{sec:global}. Here, we
  introduce the definition of coextensions of monoid functors and in particular monoid
  schemes. We show that in the latter case, localisations give us functors
  \(\coex(\cO_{X,x}, \cA_x)\to \coex(\cO_{X,y}, \cA_y)\) for \(y\leq x \in X\), which
  allows us to regard \(\coex(\cO_{X,x}, \cA_x)\) as a stack.

  Section~\ref{sec:classification_sheaves} deals with generalising the Grillet complex
  for sheaves of monoids over a poset. Among other things, we define \(Z^1(\cF, \cA)\)
  explicitly, which plays the role that \(\ker\partial^0\) in the affine setting. Using
  this, we prove the global classification theorem for sheaf coextensions.

  Section~\ref{sec:semilattice} finally deals with a special case, namely the
  coextensions of semilattices. We show that while \(\bH(M)\) is rather rich, the
  coextensions are trivial, both in the affine and non-affine setting, as might be
  expected. It serves as a little example at the end, but also deals with a relatively
  important class of monoids, as the spectrum of a monoid is a semilattice.

\section{Preliminaries on the geometry of monoids} \label{sec:spec}

  \begin{Conv}[Monoids]
    \emph{Monoid are assumed to be commutative} throughout, unless very explicitly
    stated otherwise. \emph{We do not assume the existence of an absorbing element}
    (zero in the multiplicative notation), nor is it respected by a monoid
    homomorphism, if it coincidentally exists.
  \end{Conv}

  We begin by recalling the main results of \cite{p1}, which deals with the prime
  ideals of a commutative monoid \(M\), and will play an important role throughout the
  paper.

\subsection{\(\Spec(M)\) and \(M^{sl}\)}

    Recall the following important definitions.

\subsubsection{Ideals}

      The ideal theory follows that of ring theory verbatim, with the exception that we
      do not have addition, so disregard that part in our definitions.

      A subset \(\fa\subseteq M\) of a monoid \(M\) is called an \emph{ideal} if for
      every \(a\in \fa\) and every \(m \in M\), \(am\in \fa\). To put it an other way,
      if \(\fa\) is a sub-\(M\)-set. Ideals are thus allowed to be \emph{empty}.

      An ideal \(\fa\) that is of the form
      \[ (a) \;=\; aM \;:=\; \{am\mid m\in M\} \]
      is called a \emph{principal} ideal generated by \(a\).

      An ideal \(\p\subseteq M\) is called \emph{prime} if \(1 \notin \p\) and \(ab\in
      \p\), \(b\notin \p\) implies \(a\in \p\).

      An ideal \(\fm\subseteq M\) is called \emph{maximal} if \(1 \notin \fm\) and for
      any chain of inclusions of ideals \(\fm \subsetneq \fa \subseteq M\), we have
      \(\fa = M\).

      It is not hard to see that every monoid \(M\) has a unique maximal ideal, namely
      \(M\setminus M^\times\), the set of all non-invertible elements. Indeed, a monoid
      is a group if and only if it has two ideal, being \((1)\) and \(\emptyset\). An
      other way of saying that is if \(M^\times \subseteq M\), the subgroup of
      invertible elements of \(M\), forms an ideal.

\subsubsection{\(\Spec(M)\)}

      The set of all prime ideals of a monoid is denoted by \(\Spec(M)\). It carries
      with it a natural topology, called the \emph{Zariski topology}, where open sets
      are generated by sets of the form
      \[ D(f) \;:=\; \{\p \in \Spec(M) \mid f \notin \p\}. \]

      Here is where monoid ideal-theory diverges from ring ideal-theory: The
      (set-)union of ideals is again an ideal. Moreover, unlike intersection, primeness
      is preserved under union. Meaning, \(\Spec(M)\) is a join-semilattice under
      union. The least element is the empty set and the set of of non-invertible
      elements (maximal ideal, being the union of all ideals not containing \(1\)) is
      the greatest element. As one might expect, it does have a dual operation. It is
      given by
      \[ \p\capcup \q \;:=\; \bigcup\limits_{\fr \subseteq \p \cap \q} \fr, \]
      that is, the union of all prime ideas contained in both \(\p\) and \(\q\). These
      operations are continuous in the Zariski topology, making \(\Spec(M)\) a
      topological monoid (or to be more precise, a topological semi-lattice).

\subsubsection{Semilattices and \(\Msl\)}

      It becomes evident that semilattices and idempotent monoids are an important part
      of studying monoid theory. For this, let us formally introduce them:

      A \emph{(join) semilattice} is a poset \(L\) with a least element such that every
      pair \(a, b \in L\) has a join (least upper bound) \(a \vee b\). There exists a
      dual notion of meet-semilattice, but as the theory is identical (just
      order-revered), we will not state them. Indeed:
      \begin{Conv}[Semilattice]
        By semilattice we will hence mean join-semilattice.
      \end{Conv}
      Semilattices and monoids are intimately linked: Every semilattice is a monoid
      with \(\vee\) being the operation and the least element as the unit. It is clear
      that \(a\vee a = a\). In turn, if \(m^2 = m\) holds for all \(m\in M\), we can
      define an ordering by \(m \leq n\) if and only if \(mn = n\). Under this, the
      unit becomes the least element. We can formalise this by saying that
      \begin{quote}
        The category of commutative monoids satisfying \(m^2 = m\) for all \(m\) is
        equivalent to the category \(\SL\) of semilattices.
      \end{quote}

      Let \(\I := \{0, 1\}\) be the multiplicative monoid of the field with 2 elements.
      It is naturally also a semilattice. Indeed, it is the free semilattice with one
      generator.

      We imbued \(\I\) with the topology that declares \{\(\emptyset\), \(\{1\}\),
      \(\I\)\} as open sets. One checks that \(\I \scong \Spec(\N)\) as topological
      monoids, where \(\N = \{1, t, t^2, \ldots\}\) is the free commutative monoid with
      one generator.\\[1pt]

      We have a covariant functor
      \[ \sl: \Mon \to \SL \]
      which assigns \(M/\sim\) to a monoid \(M\), where \(\sim\) is the smallest
      congruence \(\sim\) such that \(m \sim m^2\) for all \(m\). This is the universal
      semilattice quotient of \(M\) and we denote it's image by \(M^{sl}\). By Grillet
      \cite[Theorem~1.2, Ch.~III]{gr}, \(M^{sl} = M/{\sim}\) where \(a \sim b\) if and
      only if there exist \(m, n \geq 1\) and \(u, v \in M\) with \(a^m = ub\) and
      \(b^n = va\). Note that this is the exact relation we have for the basis of the
      Zariski topology, being equivalent to \(D(a) = D(b)\).\\[2pt]

      We list the following results, the proofs can be found in \cite{p1}.

      \begin{Le}[Reduction Lemma {\cite[Lemma 2.1]{p1}}] \label{lem:spec=hom}
        For any monoid \(M\) there are natural isomorphisms of topological monoids
        \[ \Spec(M) \cong \Spec(M^{sl}) \cong \Hom(M, \I), \]
        where \(M^{sl}:=M/m^2\sim m\) is the associated semilattice of \(M\). Both
        \(\Spec(M)\) and \(M^{sl}\) carry a natural topology induced by order, where
        set satisfying the property ``if \(x\in U\) and \(y\leq x\), then \(y\in U\)''
        are called open. Note, in \(M^{sl}\) we would have to take \(y\geq x\). The
        right-most side carries the subspace topology induced by the product topology
        on \(\prod_{m \in M} \I\).
      \end{Le}

      The Reduction Lemma says that the study of \(\Spec(M)\) reduces entirely to the
      study of spectra of semilattices and that the spectrum of a monoid is a
      semilattice itself.

      In the finitely generated case, we also have the following:

      \begin{Le}
        Let \(M\) be a finitely generated monoid. There is an order-reversing
        isomorphism
        \[ \Spec(M)\cong M^{sl}. \]
      \end{Le}

      We also have the following in the non-finitely generated case:

      \begin{Cor}[{\cite[Corollary 2.2]{p1}}]
        If \((M_j)_{j \in J}\) is a filtered direct system of monoids, then the
        canonical map
        \[
          \Spec\!\bigl(\colim_{j} M_j\bigr) \;\longrightarrow\; \lim_{j} \Spec(M_j)
        \]
        is a homeomorphism.
      \end{Cor}

      \begin{proof}
        Since \(\Hom(-, \I)\) converts colimits to limits, the result is immediate from
        Lemma~\ref{lem:spec=hom}.
      \end{proof}

\section{Systems of Abelian Groups over a Monoid} \label{sec:hm}

  We now introduce the algebraic structure that serves as the coefficient system for
  deformation theory of monoid schemes. The definitions in this section go back to
  Grillet \cite[p.~118]{gr}.

\subsection{The category \texorpdfstring{\(\bH(M)\)}{H(M)}}

    \begin{De}
      For a monoid \(M\), we define the category \(\bH(M)\) as follows:
      \begin{itemize}
        \item The \emph{objects} are the elements of \(M\).
        \item For \(a, b \in M\), a \emph{morphisms} \(a \to b\) is a pair \((r, a)\)
              with \(r \in M\) and \(b = ra\). We write this morphism as \(a \xto{r}
              ra\).
        \item The \emph{composite} of \(a \xto{r} ra\) and \(ra \xto{s} sra\) is
              \(a \xto{sr} sra\).
      \end{itemize}
    \end{De}

    Note that \(\bH(M)\) is a small category. The identity morphism at \(a\) is \(a
    \xto{1} a\). Composition is associative because it is the monoid structure.

    \begin{De}
      A \emph{system of abelian groups over \(M\)}, or an \emph{\(\bH(M)\)-module}, is
      a covariant functor \(\cA: \bH(M) \to \Ab\). Explicitly, \(\cA\) consists of
      \begin{itemize}
        \item an abelian group \(\cA(a)\) for each \(a \in M\),
        \item a group homomorphism \(r_* := \cA(r,a): \cA(a) \to \cA(ra)\) for each
              pair \(a, r \in M\),
      \end{itemize}
      such that the identities \(1_* = \id_{\cA(a)}\) and \((rs)_* = r_* \circ s_*\)
      (as maps \(\xymatrix{
 \cA(a) \ar@<0.3em>[r]^{(rs)_*} \ar@<-0.3em>[r]_{r_* \circ
      s_*}                                                      & \cA(rsa)
      }\) for all \(a, r, s \in M\)) hold.

      A \emph{morphism} \(\alpha: \cA \to \cB\) of systems of abelian groups over \(M\)
      is a natural transformation of functors \(\bH(M) \to \Ab\). Concretely,
      \(\alpha\) consists of group homomorphisms \(\alpha(a): \cA(a) \to \cB(a)\) for
      each \(a \in M\), such that for every \(r \in M\) the following square commutes:
      \[
        \xymatrix{
          \cA(a) \ar[r]^{r_*} \ar[d]_{\alpha(a)} & \cA(ra) \ar[d]^{\alpha(ra)} \\
          \cB(a) \ar[r]_{r_*}                    & \cB(ra)
        }
      \]
      We denote by \(\cH(M)\) the resulting category of systems of abelian groups over
      \(M\).
    \end{De}

    \begin{Rem}
      The category \(\cH(M)\) is abelian. It has limits and colimits and both are
      computed pointwise. In particular, a short exact sequence in \(\cH(M)\) is a
      sequence
      \[ 0 \to \cA \to \cB \to \cC \to 0 \]
      such that \(0 \to \cA(a) \to \cB(a) \to \cC(a) \to 0\) is exact for every \(a \in
      M\). We will talk more about it in Section~\ref{sec:long_exact}
    \end{Rem}

\subsection{Functoriality and pullback}

    \begin{De}
      A monoid homomorphism \(\varphi: M \to N\) induces a functor \(\bH(\varphi):
      \bH(M) \to \bH(N)\) in the obvious way (by sending \(a \xto{r} ra\) to
      \(\varphi(a) \xto{\varphi(r)} \varphi(ra)\)). Hence, it induces a \emph{pullback
      functor}
      \[
        \varphi^*: \cH(N) \to \cH(M), \qquad (\varphi^* \cA')(a) \;:=\;
        \cA'(\varphi(a)).
      \]
      The restriction map \((\varphi^* \cA')(r, a): \cA'(\varphi(a)) \to
      \cA'(\varphi(ra)) = \cA'(\varphi(r)\varphi(a))\) is the map \(\varphi(r)_*\) in
      \(\cA'\).
    \end{De}

\subsection{Localisation of systems of abelian groups}

    Let \(S \subseteq M\) be a submonoid, and denote by \(S^{-1}M\) the localisation of
    \(M\) at \(S\). Recall that elements of \(S^{-1}M\) are equivalence classes of
    symbols \(\tfrac{a}{s}\) with \(a \in M\) and \(s \in S\), where \(\tfrac{a_1}{s_1}
    = \tfrac{a_2}{s_2}\) in \(S^{-1}M\) if and only if there exists a \(u \in S\) such
    that \(a_1 s_2 u = a_2 s_1 u\) in \(M\). Multiplication is \(\tfrac{a_1}{s_1} \cdot
    \tfrac{a_2}{s_2} := \tfrac{a_1 a_2}{s_1 s_2}\). The canonical map \(i_S: M \to
    S^{-1}M\), \(a \mapsto \tfrac{a}{1}\), induces a pullback functor \(i_S^*:
    \cH(S^{-1}M) \to \cH(M)\).

\subsubsection{Localisation of a system of abelian groups}

      We now construct the left adjoint \(i^S_!\) of \(i_S^*\), which we call the
      \emph{localisation} of a system of abelian groups.

      For a system \(\cA \in \cH(M)\), we define
      \((S^{-1}\cA)\bigl(\tfrac{a}{s}\bigr)\) for \(\tfrac{a}{s} \in S^{-1}M\) as
      follows: The abelian group \((S^{-1}\cA)\bigl(\tfrac{a}{s}\bigr)\) consists of
      equivalence classes of pairs \((x, ss')\), with \(s' \in S\) and \(x \in
      \cA(as')\), where \((x, ss_1) \sim (x, ss_2)\) if there exist \(u\in S\) such
      that \((s_2'u)_*(x_1) = (s_1'u)_*(x_2)\). We can illustrate it in form of a
      diagram:
      \[
        \cA(ass_1') \xto{(ss_2'u)_*} \cA(ass_1's_2'u) \xfrom{(ss_1'u)_*} \cA(ass_2').
      \]
      The group operation is defined on representatives by
      \[
        \frac{x_1}{ss_1'} + \frac{x_2}{ss_2'}
        \;:=\;
        \frac{(ss_2')_*(x_1) + (ss_1')_*(x_2)}{ss_1' ss_2'},
      \]
      where both \((ss_2')_*(x_1)\) and \((ss_1')_*(x_2)\) land in the common group
      \(\cA(as_1's_2')\). One verifies in the standard way that this operation is
      well-defined on equivalence classes and gives
      \((S^{-1}\cA)\bigl(\tfrac{a}{s}\bigr)\) the structure of an abelian group, with
      zero element \(\tfrac{0}{1}\) and negatives \(-\tfrac{x}{ss'} =
      \tfrac{-x}{ss'}\).\\[1pt]

      For a morphism \(\tfrac{b}{t}: \tfrac{a}{s} \to \tfrac{ab}{st}\) in
      \(\bH(S^{-1}M)\), the induced homomorphism
      \[
        \left(\frac{b}{t}\right)_* :
        (S^{-1}\cA)\!\left(\frac{a}{s}\right) \longrightarrow
        (S^{-1}\cA)\!\left(\frac{ab}{st}\right)
      \]
      is defined on representatives by
      \[
        \left(\frac{b}{t}\right)_*\!\left(\frac{x}{ss'}\right)
        \;:=\;
        \frac{b_*(x)}{sts'},
      \]
      where \(b_*(x) \in \cA(abs')\) is the image of \(x \in \cA(as')\) under \(b_*:
      \cA(as') \to \cA(abs')\). One checks that this is compatible with the equivalence
      relation and satisfies the functoriality identities, so that \(S^{-1}\cA\) is
      indeed a system of abelian groups over \(S^{-1}M\).\\

      For a prime ideal, the complement \(S = M \setminus \p\) is a submonoid of \(M\).
      As such, we can localise with this submonoid. In this case, we use the special
      notation \(\cA_\p\) for \((M \setminus \p)^{-1}\cA\), and use the notation
      \(i_\p\), \(i_\p^*\), \(i_!^\p\) for the corresponding maps and functors.

      \begin{Rem} \label{rem:localisations_are_wtih_primes}
        Unlike for ring theory, in monoids, every localisation can be assumed to be
        done at a prime ideal. This can be done constructively. Indeed, consider the
        associated semilattice \(M^{sl}\) and consider all the classes \([m_1],\ldots,
        [m_k]\) in which \(S\) resides (has non-empty intersection). The collection of
        al such classes, and all the classes less-than-or-equal defines a submonoid,
        the complement of which is a prime ideal (such submonoids are called
        \emph{faces}). Localising at our constructed prime agrees with localising with
        \(S\).
      \end{Rem}

\subsection{The semidirect product} \label{sec:semidirect}

    We come now to an important construction, which is basically the trivial extension
    and will be our unit-extension, once we endow extensions with a symmetric
    categorical structure in~\ref{prop:coex_catgroup}.

    Let \(\CM/M\) denote the slice category over \(M\). Recall that its objects are
    monoid homomorphisms \(p: D \to M\) and morphisms are commuting triangles
    \[
      \xymatrix{
        D\ar[rr]^{f} \ar[dr]_{p} &    & D'\ar[dl]^{p'} \\
                                 & M. & 
      }
    \]
    Given \(\cA \in \cH(M)\), we may form the following important construction:

    \begin{De}
      The \emph{semidirect product} \(M \ltimes \cA\) is the monoid (see
      Lemma~\ref{lem:semidirect=monoid})
      \[ M \ltimes \cA \;:=\; \{(a, x) \mid a \in M,\; x \in \cA(a)\}, \]
      equipped with the binary operation
      \[ (a, x)(a', x') \;:=\; (aa',\; a_*(x') + a'_*(x)). \]
    \end{De}

    \begin{Le} \label{lem:semidirect=monoid}
      \(M \ltimes \cA\) is a commutative monoid.
    \end{Le}

    \begin{proof}
      Associativity follows from functoriality, that is \((rs)_* = r_* \circ s_*\), and
      the commutativity of addition. Specifically, one checks that both
      \(((a,x)(a',x'))(a'',x'')\) and \((a,x)((a',x')(a'',x''))\) equate to
      \[ \bigl(aa'a'',\; (a'a'')_*(x) + (aa'')_*(x') + (aa')_*(x'')\bigr). \]
      Since \(M\) is commutative and each \(\cA(a)\) is abelian, we have
      \[
        (a,x)(a',x') \;=\; (aa',\, a_*(x') + a'_*(x)) \;=\; (a'a,\, a'_*(x) + a_*(x'))
        \;=\; (a',x')(a,x).
      \]
      which shows that the semidirect product is commutative.

      The element \((1_M, 0_{1_M})\) acts as the identity, as using \(1_* = \id\) and
      \(a_*(0) = 0\), gives us
      \[ (1_M, 0)(a, x) \;=\; (a,\; 1_*(x) + a_*(0)) \;=\; (a, x). \qedhere \]
    \end{proof}

\subsection{The abelian group objects description}

    Recall that a \emph{group object} in a category \(\cC\) is an object \(G \in \cC\),
    together with morphisms \(\mu: G \times G \to G\) (multiplication), \(\nu: G \to
    G\) (inverse), and \(e: \mathbf{1} \to G\) (unit, where \(\mathbf{1}\) is the
    terminal object of \(\cC\)), satisfying the usual group axioms expressed via
    commutative diagrams. We say it is \emph{abelian} if it additionally satisfies the
    commutativity constraint \(\mu \circ \tau = \mu\), where \(\tau: G \times G \to G
    \times G\) is the swap.

    The terminal object in the slice category \(\CM/M\) is \(\id_M: M \to M\) and fibre
    products are taken over \(M\). An abelian group object \(p: D \to M\) thus carries
    a zero section \(0_D: M \to D\), a fibrewise addition \(+_D: D \times_M D \to D\),
    and a fibrewise negation, all compatible with \(p\).

    \begin{Pro}
      The category \(\cH(M)\) is equivalent to the category of abelian group objects in
      \(\CM/M\).
    \end{Pro}

    \begin{proof}
      We proceed with the usual approach for proving something like this: We build
      functors on both sides and show that these are mutually inverse.

      \medskip
      Let \(\cA \in \cH(M)\) be an abelian system. The projection
      \[ p_\cA: M \ltimes \cA \to M, \qquad (a,x) \mapsto a, \]
      is an object of \(\CM/M\). We equip it with the structure of an abelian group
      object via
      \[
        \begin{array}{rl}
          0_\cA: M \to M \ltimes \cA ,                                        & \quad a \mapsto (a, 0_{a}), \\
          +_\cA: (M \ltimes \cA) \times_M (M \ltimes \cA) \to M \ltimes \cA , & \quad (a,x),(a,x') \mapsto (a, x + x').
        \end{array}
      \]
      The map \(0_\cA\) is a monoid homomorphism as
      \[
        0_\cA(a) \cdot 0_\cA(a') \;=\; (a,0)(a',0)
        \;=\; (aa', a_*(0) + a'_*(0)) \;=\; (aa', 0) \;=\; 0_\cA(aa').
      \]
      Likewise, the map \(+_\cA\) is a monoid homomorphism on the fibre product, with
      fibrewise commutativity and associativity inherited from \(\cA(a)\). The
      negatives are given by \((a,x) \mapsto (a,-x)\). These maps are all morphisms
      over \(p_\cA\), making it an abelian group object in \(\CM/M\).

      \medskip
      Conversely, let \(p: D \to M\) be an abelian group object in \(\CM/M\), with zero
      section \(0_D: M \to D\) and fibrewise addition \(+_D: D \times_M D \to D\). For
      each \(a \in M\), set \(\cA(a) := p^{-1}(a)\) and define the abelian group
      structures through \(+_D\). The role of the unit is played by \(0_D(a)\).

      The actions are defined as follows: For \(r \in M\) and \(x \in \cA(a)\), define
      \[ r_*(x) \;:=\; 0_D(r) \cdot x \;\in \cA(ra). \]
      To see that this makes sense, observe that \(p: D\to M\) is a monoid homomorphism
      and hence,
      \[ p(0_D(r) \cdot x) \;=\; p(0_D(r)) \cdot p(x) \;=\; r \cdot a \;=\; ra. \]
      It follows that \(0_D(r) \cdot x \in p^{-1}(ra) = \cA(ra)\). Note that \(p(x) =
      a\) by the definition of \(\cA(a)\) as \(p^{-1}(a)\).

      The functoriality identities hold as well. We have \(1_* = \id\) as \(0_D(1_M) =
      1_D\), and \((rs)_* = r_* \circ s_*\) since
      \[
        (rs)_*(x) \;=\; 0_D(rs) \cdot x \;=\; (0_D(r) \cdot 0_D(s)) \cdot x
        \;=\; 0_D(r) \cdot (0_D(s) \cdot x) \;=\; r_*(s_*(x)).
      \]
      Finally, \(r_*: \cA(a) \to \cA(ra)\) is a group homomorphism as the group object
      axioms in \(\CM/M\) require that \(+_D\) is a morphism of monoids over \(M\),
      which forces multiplication by any fixed element \(0_D(r) \in D\) to respect
      \(+_D\). Concretely,
      \[
        r_*(x +_D x') \;=\; 0_D(r) \cdot (x +_D x')
        \;=\; (0_D(r) \cdot x) +_D (0_D(r) \cdot x')
        \;=\; r_*(x) +_D r_*(x'),
      \]
      where the middle equality holds because \(+_D\) is a monoid morphism.

      \medskip
      To see that they are mutually inverse, take \(\cA \in \cH(M)\) and define the
      surjection \(p_\cA: M \ltimes \cA \to M\). We then form the abelian group object
      and take fibres \(p_\cA^{-1}(a) = \cA(a)\). This recovers the original \(r_*\)
      since
      \[
        0_{p_\cA}(r) \cdot (a,x) \;=\; (r, 0) \cdot (a, x) \;=\; (ra,\; r_*(x) +
        a_*(0)) \;=\; (ra,\, r_*(x)).
      \]

      On the other hand, let \(p: D \to M\) be a group object and form \(\cA(a) =
      p^{-1}(a)\). We construct an isomorphism
      \[ \phi: M \ltimes \cA \to D, \qquad (a, x) \mapsto x \]
      in \(\CM/M\) compatible with the abelian group object structure. The only part
      that needs any convincing is that its a homomorphism, but this can be verified by
      \[
        \phi((a,x)(a',x')) \;=\; a_*(x') +_D a'_*(x)
        \;=\; 0_D(a) \cdot x' +_D 0_D(a') \cdot x \;=\; \phi(a, x)\phi(a', x').
      \]
      Thus, the two constructions are mutually inverse, up to natural isomorphism, and
      we are done.
    \end{proof}

\section{The Category of Coextensions} \label{sec:coext}

\subsection{Definition of coextensions and basic properties}

    \begin{De}[{\cite[Section~V.3]{gr}}] \label{def:coext}
      Let \(M\) be a monoid and \(\cA\) a system of abelian groups over \(M\). A
      \emph{group coextension} of \(M\) by \(\cA\) is a surjective monoid homomorphism
      \(\pi: N \to M\), together with an action of \(\cA(m)\) on \(\pi^{-1}(m)\) for
      each \(m \in M\), satisfying the following two conditions:
      \begin{itemize}
        \item[(i)] For every \(n, n'\in \pi^{-1}(m)\), there exists a unique \(x \in
              \cA(m)\) such that \(n' = x \pt n\).
        \item[(ii)] For every \(x \in \cA(m)\), \(n\in \pi^{-1}(m)\), \(y \in
              \cA(m')\), \(n'\in\pi^{-1}(m')\), we have
              \[ (x \pt n)(y \pt n') \;=\; \bigl(m'_*(x) + m_*(y)\bigr) \pt (nn'). \]
      \end{itemize}
      The first property will be referred to as \emph{regularity} and the second one as
      \emph{Compatibility}.

      \noindent We denote such a coextension by \(0 \to \cA \to N \xto{\pi} M \to 0\).
    \end{De}

    \(\cA(m)\) acting on \(\pi^{-1}(m)\), of course, means \(0 \pt n = n\) and \((x +
    y) \pt n = x \pt (y \pt n)\) for all \(x, y \in \cA(m)\), \(n \in \pi^{-1}(m)\).

    It is not hard to see that the coextensions theory of monoids generalises the
    coextension theory of abelian groups, see
    Proposition~\ref{prop:monoig_coex_generalises_group_coex}.

    \begin{Le} \label{lem:exact_at_M}
      For all \(n \in N\), \(x \in \cA(\pi(n))\), we have \(\pi(x\bullet n) = \pi(n)\).
    \end{Le}

    \begin{proof}
      This is just the definition of a coextension if we take \(m = \pi(n)\), as
      \(\cA(\pi(n))\) acts on \(\pi^{-1}(\pi(n))\) and clearly \(n\in
      \pi^{-1}(\pi(n))\). Hence, \(x\bullet n\in \pi^{-1}(\pi(n))\).
    \end{proof}

    \begin{Rem}
      The notation \(0 \to \cA \to N \xto{\pi} M \to 0\) should be regarded with care:
      It does \emph{not} assert the existence of an injective map \(\cA \to N\), since
      \(\cA\) is not a monoid (or even set).
    \end{Rem}

    \begin{Le} \label{lem:inv_ext}
      Let \(0 \to \cA \to N \xto{\pi} M \to 0\) be a coextension.
      \begin{enumerate}
        \item An element \(n \in N\) is invertible in \(N\) if and only if \(\pi(n)\)
              is invertible in \(M\).
        \item The map \(f: \cA(1_M) \to N\), given by \(x \mapsto x \pt 1_N\), is a
              monoid homomorphism.
        \item We have a short exact sequence of abelian groups
              \[ 0 \to \cA(1_M) \xto{f} N^\times \xto{\pi} M^\times \to 0. \]
      \end{enumerate}
    \end{Le}

    \begin{proof}
      (1) Suppose \(\pi(n) = m\) is invertible in \(M\). Since \(\pi\) is surjective,
      we can pick \(n' \in N\) with \(\pi(n') = m^{-1}\). Then \(\pi(nn') = 1_M =
      \pi(1_N)\) and by Definition~\ref{def:coext}(i) (regularity), there exists a
      unique \(x \in \cA(1_M)\) with \(nn' = x \pt 1_N\). Consider \((m^{-1})_*(-x) \pt
      n' \in \pi^{-1}(m^{-1})\). Definition~\ref{def:coext}(ii) (compatibility) now
      states
      \begin{align*}
        n \cdot \bigl((m^{-1})_*(-x) \pt n'\bigr)
                                                  & = (0 \pt n) \cdot \bigl((m^{-1})_*(-x) \pt n'\bigr)       \\
                                                  & = \bigl(m_*((m^{-1})_*(-x)) + m^{-1}_*(0)\bigr) \pt (nn') \\
                                                  & = (-x + 0) \pt (x \pt 1_N)                                \\
                                                  & = (-x + x) \pt 1_N = 1_N.
      \end{align*}
      So \(n\) is left-invertible. Since \(N\) is commutative, it is also
      right-invertible.

      (2) Setting \(n = n' = 1_N\), \(m = m' = 1_M\) in Definition~\ref{def:coext}(ii)
      and using \(1_* = \id\) we get,
      \[
        (x \pt 1_N)(y \pt 1_N) \;=\; (1_*(x) + 1_*(y)) \pt 1_N \;=\; (x + y) \pt 1_N.
      \]
      In other words, \(f(x)f(y) = f(x+y)\), and \(f(0) = 0 \pt 1_N = 1_N\), so \(f\)
      is a monoid homomorphism.

      (3) To see the exactness at \(\cA(1_M)\), observe that if \(x \pt 1_N = 1_N = 0
      \pt 1_N\), then \(x = 0\) by the uniqueness in Definition~\ref{def:coext}(i), so
      \(f\) is injective.

      Exactness at \(N^\times\) follows by (1). We have \(\ker(\pi|_{N^\times}) =
      \pi^{-1}(1_M) \cap N^\times\) and thus, every element of \(\pi^{-1}(1_M)\) is of
      the form \(x \pt 1_N = f(x)\), so \(\ker(\pi|_{N^\times}) =
      \operatorname{im}(f)\).

      Exactness at \(M^\times\) follows from the surjectivity of \(\pi\) together with
      (1).
    \end{proof}

    A groupoid is called \emph{simply connected} if every pair of objects is connected
    by exactly one isomorphism. Any such groupoid is equivalent to the trivial
    groupoid, meaning the groupoid with a single object and single isomorphism.

    \begin{Pro} \label{prop:monoig_coex_generalises_group_coex}
      Let \(G\) be an abelian group.
      \begin{enumerate}
        \item The category \(\bH(G)\) is a simply connected groupoid.
        \item The functor \(\cH(G) \to \Ab\), induced by evaluating at \(e_G\), is an
              equivalence of categories. In particular, any system of abelian groups on
              \(G\) is constant.
        \item If \(0 \to \cA \to N \xto{\pi} G \to 0\) is a coextension, \(N\) is
              a group and \(0 \to \cA(e_G) \to N \xto{\pi} M \to 0\) is a short exact
              sequence of abelian groups, and visa verse.
      \end{enumerate}
    \end{Pro}

    \begin{proof}
      Assertion (1) follows since for any \(a.b\in G\), there exist a unique \(c\in G\)
      with \(b=ac\). For (2), recall that a simply connected groupoid is equivalent to
      a trivial groupoid, hence the result. The first statement of 3) follows directly
      from Lemma~\ref{lem:inv_ext}.
    \end{proof}

    \noindent It follows that Definition~\ref{def:coext} agrees with the classical
    definition of coextensions.

\subsection{The semilattice quotient of a coextension}

    \begin{Le}
      \label{lem:sl_coext} Let \(0 \to \cA \to N \xto{\pi} M \to 0\) be a coextension.
      The induced map
      \[ N^{sl} \longrightarrow M^{sl} \]
      is an isomorphism of semilattices.
    \end{Le}

    \begin{proof}
      It suffices to show that any monoid homomorphism \(h: N \to L\), where \(L\)
      satisfies \(a^2 = a\) for all \(a\), factors through \(\pi: N \to M\). For this,
      it suffices to show by Lemma~\ref{lem:exact_at_M} that \(h(x \pt n) = h(n)\) for
      all \(n \in N\) and \(x \in \cA(\pi(n))\).

      \textbf{Step 1: \(h(n) \leq h(x \pt n)\).} This follows from Definition
      ~\ref{def:coext}(ii) with \(n = n'\) and using commutativity, as we have
      \[
        (x \pt n)((-x) \pt n) \;=\; \bigl(\pi(n)_*(x) + \pi(n)_*(-x)\bigr) \pt n^2
        \;=\; 0 \pt n^2 \;=\; n^2.
      \]
      Since \(h\) is a homomorphism and every element of \(L\) is an idempotent, we
      deduce
      \[ h(x \pt n) \cdot h((-x) \pt n) \;=\; h(n)^2 \;=\; h(n), \]
      so \(h(n) \leq h(x \pt n)\) in the semilattice order.

      \textbf{Step 2: \(h(x \pt n) \leq h(n)\).} Let both factors equal to \(x \pt n\)
      in Definition~\ref{def:coext}(ii) to obtain
      \[ (x \pt n)^2 \;=\; \bigl(2\pi(n)_*(x)\bigr) \pt n^2. \]
      On the other hand
      \[ n(2x\pt n) \;=\; (0\pt n)(2x\pt n) \;=\; (\pi(n)_*(2x)\pt n^2 \]
      It follows that
      \[ (x \pt n)^2 =n(2x\pt n) \]
      Acting with \(h\) now gives us
      \[ h(x \pt n)=h(x \pt n)^2=h(n)h(2x\pt n)) \]
      A similar argument to Step 1 gives us
      \[ h(x \pt n)\leq h(n) \]
      Combining Steps 1 and 2 gives the desired result.
    \end{proof}

    As a direct corollary of Lemma~\ref{lem:spec=hom} and Lemma~\ref{lem:sl_coext}, we
    obtain the following:

    \begin{Cor} \label{cor:spec_coext}
      If \(0 \to \cA \to N \xto{\pi} M \to 0\) is a coextension, then
      \[ \pi^{-1}: \Spec(M) \xrightarrow{\;\sim\;} \Spec(N) \]
      is a homeomorphism of topological spaces.
    \end{Cor}

    Thanks to Corollary~\ref{cor:spec_coext}, we can and will identify \(\Spec(N)\)
    with \(\Spec(M)\) via \(\pi^{-1}\).

    \begin{Le} \label{lem:localisation_coext}
      Let \(0 \to \cA \to N \xto{\pi} M \to 0\) be a coextension. For any prime ideal
      \(\p \in \Spec(M)\), the localisation at \(M \setminus \p\) yields an induced
      coextension of monoids
      \[ 0 \to \cA_\p \to N_\p \to M_\p \to 0. \]
    \end{Le}

    \begin{proof}
      The localisation \(N_\p = (N \setminus \pi^{-1}(\p))^{-1}N\) maps surjectively to
      \(M_\p = (M \setminus \p)^{-1}M\) via the localisation of \(\pi\). For
      \(\dfrac{a}{s} \in M_\p\), the fibre \(\pi_\p^{-1}\left(\dfrac{a}{s}\right)\)
      consists of all fractions \(\dfrac{n}{t} \in N_\p\) with \(\dfrac{\pi(n)}{\pi(t)}
      = \dfrac{a}{s}\). Given \(\dfrac{x}{(s \cdot \pi(t'))} \in
      \cA_\p\left(\dfrac{a}{s}\right)\) with \(x \in \cA(a \cdot \pi(t'))\) and \(t'
      \in N \setminus \pi^{-1}(\p)\), and \(\dfrac{n}{t} \in
      \pi_\p^{-1}\left(\dfrac{a}{s}\right)\), define:
      \[
        \frac{x}{s \cdot \pi(t')} \pt \frac{n}{t} \;:=\; \frac{x \pt (t' n)}{t' t}.
      \]
      One verifies this is well defined and satisfies the regularity and compatibility
      conditions of Definition~\ref{def:coext}.
    \end{proof}

    \begin{Cor}
      Let \(M\) be a finitely generated monoid and \(0 \to \cA \to N \xto{\pi} M \to
      0\) be a coextension. For any multiplicative subset \(S \subseteq M\), the
      localisation with \(S\) yields an induced coextension of monoids
      \[ 0 \to S^{-1}\cA \to S^{-1}N \to S^{-1}M \to 0. \]
    \end{Cor}

    \begin{proof}
      This is essentially Lemma~\ref{lem:localisation_coext} and
      Remark~\ref{rem:localisations_are_wtih_primes}.
    \end{proof}

\subsection{Morphisms of coextensions}

    \begin{De} \label{de:map_coextensions}
      Let \(\cE = (0 \to \cA \to N \xto{\pi} M \to 0)\) and \(\cE' = (0 \to \cA' \to N'
      \xto{\pi'} M' \to 0)\) be coextensions. A \emph{morphism} \(\cE \to \cE'\)
      \[
        \xymatrix{
          0\ar[r] & \cA\ar[r]\ar@{=>}[d]^\alpha & N\ar[r]^\pi\ar[d]^h & M\ar[r]\ar[d]^\varphi & 0 \\
          0\ar[r] & \cA'\ar[r]                  & N'\ar[r]^{\pi'}     & M'\ar[r]              & 0
        }
      \]
      is a triple \((\alpha, h, \varphi)\), where \(h: N \to N'\), \(\varphi: M \to
      M'\) are monoid homomorphisms and \(\alpha: \cA \to \varphi^*(\cA')\) is a
      morphism in \(\cH(M)\) (which is a natural transformation, hence the distinct
      notation), satisfying:
      \begin{itemize}
        \item[(i)] \(\varphi \circ \pi = \pi' \circ h\) (i.e.\ the obvious square
              commutes);
        \item[(ii)] \(h(x \pt n) = \alpha(x) \pt h(n)\) for all \(n \in N\) and \(x \in
              \cA(\pi(n))\).
      \end{itemize}
      The category of all coextensions is denoted \(\coex\).
    \end{De}

    To clarify notation, by \(\alpha(x)\) we mean the following: \(\alpha:\cA\to
    \varphi^*(\cA')\) is a natural transformation. In particular, we have a group
    homomorphism \(\alpha(\pi(n)):\cA(\pi(n))\to \varphi^*(\cA')(\pi(n))\). Hence, we
    set \(\alpha(x) := \alpha(\pi(n))(x)\in \varphi^*(\cA')(\pi(n))\).

    Note also that injectivity, surjectivity and isomorphism of \(\alpha\) means the
    respective term pointwise.

    \begin{Le}[Short Five Lemma for Coextensions]
      \label{lem:short5} Let \((\alpha, h, \varphi): \cE \to \cE'\) be a morphism of
      coextensions. If \(\varphi\) and \(\alpha: \cA \to \varphi^*(\cA')\) are
      isomorphisms, then \(h: N \to N'\) is an isomorphism.
    \end{Le}

    \begin{proof}
      The proof is entirely analogous to the classical Short Five Lemma for module
      extensions.

      We start by proving that its injective. If \(h(n) = h(n')\), then \(\pi'(h(n)) =
      \pi'(h(n'))\). So \(\varphi(\pi(n)) = \varphi(\pi(n'))\) and since \(\varphi\) is
      injective, we have \(\pi(n) = \pi(n')\). The regularity condition of coextensions
      implies the exists of \(x \in \cA(\pi(n))\) with \(n' = x \pt n\). By
      Definition~\ref{de:map_coextensions}~(ii), we have \(h(n) = h(n') = h(x \pt n) =
      \alpha(x) \pt h(n)\), so \(\alpha(x) \pt h(n) = 1 \pt h(n)\) and regularity now
      implies \(\alpha(x) = 0\). Since \(\alpha\) is injective, we deduce that \(x =
      0\) and thus \(n = n'\).

      To see that its surjective, take \(n' \in N'\) and consider \(\pi'(n')\). There
      exists an element \(n \in N\) with \(\varphi(\pi(n)) = \pi'(n')\) since both
      \(\varphi\) and \(\pi\) are surjective. The commutativity of the right square now
      gives us \(\pi'(h(n)) = \varphi(m) = \pi'(n')\). Regularity of the second
      coextension now implies the existence of \(y' \in \cA'(\pi'(n'))\) with \(n' = y'
      \pt h(n)\). Since \(\alpha\) is surjective, write \(y' = \alpha(y)\) for some \(y
      \in \cA(m)\). Then \(n' = \alpha(y) \pt h(n) = h(y \pt n)\) by
      Definition~\ref{de:map_coextensions}~(ii).
    \end{proof}

\subsection{Split coextensions}

    \begin{De}
      A coextension \(0 \to \cA \to N \xto{\pi} M \to 0\) is \emph{split} if there
      exists a monoid homomorphism \(\mu: M \to N\) with \(\pi \circ \mu = \id_M\).
    \end{De}

    \begin{Rem}
      For the semidirect product with the natural projection \(\pi: M \ltimes \cA \to
      M,\; (a, x) \to a\) and fibre action \(z \pt (x, a) = (z + x, a)\) for \(z \in
      \cA(a)\), the map \(\mu(a) = (0, a)\) defines a splitting.
    \end{Rem}

    \begin{De} \label{def:derivation}
      Let \(\cA\in \cH(M)\) be a system of abelian groups. A \emph{derivation}
      \(\partial: M \to \cA\) is a function (see Remark~\ref{rem:function_N_A})
      satisfying
      \[ \partial(mm') \;=\; m'_*\partial(m) + m_*\partial(m'), \]
      for all \(m, m' \in M\).
    \end{De}

    \begin{Rem} \label{rem:function_N_A}
      Since \(\cA\) is not a monoid but a functor \(\bH(M) \to \Ab\), a ``function \(M
      \to \cA\)'' has to be understood as a collection of functions which associates to
      each element \(m \in M\), an element \(\partial(m) \in \cA(m)\). Thus
      \(\partial\) is a section of the bundle \(\prod\limits_{m \in M} \cA(m)\). Of
      course, we can extend this to a derivation over \(N\) via \(\pi\).
    \end{Rem}

    \begin{Le}
      \label{lem:split} The following are equivalent for a coextension \(0 \to \cA \to
      N \xto{\pi} M \to 0\):
      \begin{enumerate}
        \item The coextension is split.
        \item There exists an \(h: N \xrightarrow{\sim} \cA \rtimes M\), such that
              \((\id_\cA, h, \id_M)\) is an isomorphism of coextensions.
        \item There exists a derivation \(\partial: N \to \pi^*(\cA)\), given by
              \(\partial(x \pt n) = x + \partial(n)\).
      \end{enumerate}
    \end{Le}

    \begin{proof}
      \((1) \Rightarrow (3)\): Let \(\mu: M \to N\) be a monoid homomorphism with \(\pi
      \mu = \id_M\) and \(n \in N\). Since \(n\) and \(\mu(m)\) both lie in
      \(\pi^{-1}(m)\), where \(m = \pi(n)\), there must exist a unique element in
      \(\cA(m)\), let us just call it the symbol \(\partial(n)\), with \(n =
      \partial(n) \pt \mu(m)\) (regularity). For \(n, n' \in N\) with \(m = \pi(n)\),
      \(m' = \pi(n')\), we have
      \begin{eqnarray*}
        nn' & = & (\partial(n) \pt \mu(m))(\partial(n') \pt \mu(m'))         \\
            & = & (m'_*\partial(n) + m_*\partial(n')) \pt \mu(mm')           \\
            & = & (m'_*\partial(n) + m_*\partial(n')) \pt \mu(\pi(n)\pi(n')) \\
            & = & (m'_*\partial(n) + m_*\partial(n')) \pt \mu(\pi(nn')).
      \end{eqnarray*}
      Our definition (symbol) of \(\partial\) now implies
      \[ \partial(nn') \;=\; m'_*\partial(n) + m_*\partial(n'). \]
      Next, let \(x \in \cA(m)\). We have
      \[
        x \pt n \;=\; x \pt (\partial(n) \pt \mu(m)) \;=\; (x + \partial(n)) \pt \mu(m)
      \]
      (where the second equation is just the action), giving \(\partial(x \pt n) = x +
      \partial(n)\).

      \((3) \Rightarrow (2)\): Define \(h: N \to \cA \rtimes M\) by \(h(n) =
      (\partial(n), \pi(n))\). Then
      \begin{eqnarray*}
        h(nn') & = & (\partial(nn'), \pi(nn'))                                      \\
               & = & (\pi(n')_*\partial(n) + \pi(n)_*\partial(n'),\; \pi(n)\pi(n')) \\
               & = & h(n) \cdot h(n').
      \end{eqnarray*}
      It is compatible with the action as \(h(x \pt n) = (x + \partial(n), \pi(n)) = x
      \pt h(n)\), and thus, \((\id_\cA, h, \id_M)\) is a morphism of coextensions. The
      fact that its an isomorphism follows from Lemma~\ref{lem:short5}.

      \((2) \Rightarrow (1)\): The semidirect product always splits via \(\mu(a) = (0,
      a)\).
    \end{proof}

\subsection{Fixed-monoid categories and the groupoid structure}

    For a fixed monoid \(M\) and a system of abelian groups \(\cA \in \cH(M)\), let
    \(\coex(M, \cA)\) denote the full subcategory of coextensions \(\coex\) with
    objects \(0 \to \cA \to N \xto{\pi} M \to 0\) and morphisms \((\id_\cA, h,
    \id_M)\).

    In other words, the extremities are fixed, but the central part may vary.

    \begin{Pro}
      \(\coex(M, \cA)\) is a groupoid. The semidirect product \(\cA \rtimes M\) defines
      a distinguished base object. An object is isomorphic to \(\cA \rtimes M\) if and
      only if it is split.
    \end{Pro}

    \begin{proof}
      Every morphism \((\id_\cA, h, \id_M)\) in \(\coex(M, \cA)\) is an isomorphism by
      Lemma~\ref{lem:short5}. The remaining claim follows from Lemma~\ref{lem:split}.
    \end{proof}

    The set of isomorphism classes of \(\coex(M, \cA)\) is denoted \(\pi_0(\coex(M,
    \cA))\).

\subsection{Pullback along a monoid homomorphism}

    Let \(\varphi: N \to M\) be a monoid homomorphism and \(\cE = (0\to \cA\to L
    \xto{\pi} M\to 0)\in \coex(M)\) a coextension of \(M\). We can form the pullback of
    \(L\xto{\pi} M \xleftarrow{\varphi} N\), meaning
    \[
      \xymatrix{
        0 \ar[r] & \cA \ar[r] & L \ar[r]^{\pi}                                   & M \ar[r]                & 0 \\
                 &            & \square \ar@{..>}[r]_{\pi_1}\ar@{..>}[u]^{\pi_2} & N\ar[r]\ar[u]_{\varphi}
                 & 0,
      }
    \]
    in the category of monoids. This is given by
    \[ K \;:=\; \{(n, l) \in N \times L \mid \varphi(n) \;=\; \pi(l)\}. \]
    For \(n \in N\) and \(x \in (\varphi^*\cA)(n) = \cA(\varphi(n))\), define \(x \pt
    (n, l) = (n, x \pt l)\). This gives a coextension
    \[ \varphi^*(\cE) \;=\; (0 \to \varphi^*\cA \to K \xto{\pi_1} N \to 0). \]
    \begin{De}
      The above construction defines the \emph{pullback} functor
      \[ \varphi^*: \coex(M) \to \coex(N). \]
    \end{De}

    \begin{Le}
      The coextension \(\varphi^*(\cE)\) splits if and only if there exists a monoid
      homomorphism \(\psi: N \to L\) with \(\pi \circ \psi = \varphi\).
    \end{Le}

    \begin{proof}
      A splitting \(\mu: N \to K\) of \(\pi_2\) corresponds exactly to a pair \((\id_N,
      \psi)\) where \(\psi = \pi_2 \circ \mu: N \to L\) satisfies \(\pi \circ \psi =
      \varphi \circ \id_N = \varphi\).
    \end{proof}

\subsection{Push-forward along a morphism of systems}

    Let \(\alpha: \cA \to \cB\) be a morphism in \(\cH(M)\) and \(\cE = (0 \to \cA \to
    N \xto{\pi} M \to 0) \in \coex(M)\). We want to create a type of ```push-forward'''
    for \(\cB\xleftarrow{\alpha} \cA \to N\) in \(\coex(M)\), meaning:
    \[
      \xymatrix{
        0 \ar[r] & \cA \ar[r] \ar[d]_{\alpha} & N \ar[r]^{\pi} \ar@{..>}[d]^{\rho_2} & M \ar[r] \ar[d]^{\id} & 0 \\
        0\ar[r]  & \cB \ar@{..>}[r]_{\rho_1}  & \square \ar@{..>}[r]_{\sigma}        & M \ar[r]
                 & 0.
      }
    \]
    Let \(K\) be the set of equivalence classes of pairs \((y, n)\) with \(y \in
    \cB(\pi(n))\), \(n \in N\), where \((y, n) \sim (y', n')\) if
    \begin{itemize}
      \item \(\pi(n) = \pi(n')\),
      \item there exists \(x \in \cA(\pi(n))\) with \(x \pt n = n'\) and \(y -
            \alpha(x) = y'\).
    \end{itemize}
    Write \([y, n]\) for the class of \((y, n)\). We can define the structure of a
    commutative monoid on \(K\) by declaring
    \[ [y, n][z, n'] \;=\; [\pi(n')_*(y) + \pi(n)_*(z), nn']. \]
    We can define a monoid homomorphism \(\sigma: K \to M\), given by \(\sigma[y, n] =
    \pi(n)\). Furthermore, we have an action of \(\cB\) on the fibres of \(\sigma\)
    given by \(u \pt [y, n] = [u + y, n]\), for \(u \in \cB(\pi(n))\). This defines a
    coextension \((0 \to \cB \to K \xto{\sigma} M \to 0)\) of \(M\).

    \begin{De}
      The above construction defines the \emph{push-forward} functor
      \[ \alpha_*: \coex(M, \cA) \to \coex(M, \cB). \]
    \end{De}

    \begin{Le}
      The coextension \(\alpha_*(\cE)\) splits if and only if there exists a derivation
      \(\partial: N \to \pi^*(\cB)\) satisfying \(\partial(x \pt n) = \alpha(x) +
      \partial(n)\) for all \(n \in N\) and \(x \in \cA(\pi(n))\).
    \end{Le}

    \begin{proof}
      Suppose \(\alpha_*(\cE)\) splits. Take the derivation \(\partial': K \to \cB\)
      provided by Lemma~\ref{lem:split}(3). The composite \(\partial = \partial' \circ
      g\), where \(g: N \to K\), \(g(n) = [0, n]\), satisfies
      \[
        \partial(x \pt n) \;=\; \partial'([\alpha(x), n])
        \;=\; \alpha(x) + \partial'([0, n]) \;=\; \alpha(x) + \partial(n).
      \]
      Conversely, given a derivation \(\partial\) as in the condition of the lemma, we
      define \(\partial'([y, n]) = y + \partial(n)\). One can check that its
      well-defined and satisfies the derivation condition for \(\alpha_*(\cE)\).
    \end{proof}

\subsection{The Baer sum and the categorical group structure of \(\coex(M, \cA)\)}

    The aim of this subsection is to show that \(\coex(M, \cA)\) carries the structure
    of a \emph{symmetric categorical group} (also called a Picard groupoid). This is a
    groupoid equipped with a symmetric monoidal structure in which every object is
    invertible (in the categorical sense, meaning up to isomorphism).

    \begin{De}
      Given two coextensions \(\cE, \cE'\) of \(M\) by \(\cA\), their \emph{Baer sum}
      is
      \[ \cE + \cE' \;:=\; \nabla_*(\Delta^*(\cE \times \cE')), \]
      where
      \[
        \cE \times \cE' \;=\; (0 \to \cA \oplus \cA \to N \times N' \xto{\pi \times
        \pi'} M \times M \to 0),
      \]
      and
      \[
        \begin{array}{rcll}
          \Delta: M \to M \times M,       &  & a \mapsto (a, a)     & \,\text{(diagonal)}\, , \\
          \nabla: \cA \oplus \cA \to \cA, &  & (x, y) \mapsto x + y & \,\text{(fold map)}\, .
        \end{array}
      \]
    \end{De}

    To put it an other way, \(\Delta^*(\cE \times \cE')\) is the coextension of \(M\)
    by \(\cA \oplus \cA\) with underlying monoid
    \[ \{(m, n, n') \mid m \in M,\, \pi(n) \;=\; \pi'(n') \;=\; m\}, \]
    and \(\nabla_*\) pushes this forward along the fold map to produce a coextension of
    \(M\) by \(\cA\).

    \begin{Pro} \label{prop:coex_catgroup}
      The Baer sum makes \(\coex(M, \cA)\) into a symmetric categorical group in the
      following way:
      \begin{itemize}
        \item The unit object is the semidirect product \(M \ltimes \cA\).
        \item The inverse of a coextension \(\cE = (0 \to \cA \to N \xto{\pi} M \to
              0)\) is given by \(\cE^{-1} := \nu_*(\cE)\), where \(\nu: \cA \to \cA\)
              is the negation \(x \mapsto -x\).
        \item The symmetric monoidal structure is given by the Baer sum, mapping
              \[ (\cE, \cE')\to \cE + \cE' \xrightarrow{\sim} \cE' + \cE, \]
              where the swap of the two factors is what makes it symmetric.
      \end{itemize}
      In particular, both \(\pi_0(\coex(M, \cA))\)and \(\pi_1(\coex(M, \cA)) := \Aut(M
      \ltimes \cA)\) are abelian groups.
    \end{Pro}

    \begin{proof}
      We give an abridged proof, as verifying everything in detail takes a fair amount
      of writing, but is relatively clear: Associativity, meaning a natural isomorphism
      \[ (\cE_1 + \cE_2) + \cE_3 \xrightarrow{\sim} \cE_1 + (\cE_2 + \cE_3) \]
      follows from the fact that everything in the construction of the Baer sum is
      associativity. More precisely, both the fold map \(\nabla\) and the fibre product
      over \(M\). Moreover, both sides have the underlying set
      \[ \{(n_1, n_2, n_3) \mid \pi(n_i) \;=\; m\} \]
      with the same multiplication law, making \((- + -)\) a strict monoidal structure.

      That the semi-direct product \(M \ltimes \cA\) is the unit, meaning
      \[ \cE + (M \ltimes \cA) \cong \cE \]
      can be verified by the isomorphism
      \[ [y, (n, (0, m))] \mapsto [y, n] \]
      from \(\nabla_* \Delta^*(N \times (M \ltimes \cA))\) to \(N\), which one checks
      respects the coextension structure.

      The inverse axiom will follow from the classification
      (Theorem~\ref{thm:classification}) below. Specifically, to the isomorphism
      classes, \([\cE] + [\cE^{-1}]\) corresponds to \(f + (-f) = 0\) in \(\sD^1(M,
      \cA)\), with \(0\in \sD^1(M,\cA)\) corresponding to the semi-direct product.

      The symmetry
      \[ \cE + \cE' \cong \cE' + \cE \]
      follows since we have an isomorphism of coextensions
      \[ \nabla_*\Delta^*(N \times N') \simeq \nabla_*\Delta^*(N' \times N) \]
      given via the swap
      \[
        [\pi(n')_*(x) + \pi(n)_*(y), nn'] \leftrightarrow [\pi(n)_*(y) + \pi(n')_*(x),
        n'n],
      \]
      since \(\cA\) is abelian and \(M\) commutative.
    \end{proof}

\section{The Grillet Complex and Classification} \label{sec:grillet}

\subsection{The truncated Grillet complex}

    \begin{De}[{\cite[Section~XII.5]{gr}}] \label{def:grillet}
      Let \(M\) be a commutative monoid and \(\cA \in \cH(M)\). Define cochain groups:
      \begin{align*}
        \sC^0(M, \cA) & = \{g: M \to \cA \text{ with } g(a)\in \cA(a) \}                                 \\
        \sC^1(M, \cA) & = \{f: M\times M \to \cA \text{ with } f(a,b)\in \cA(ab) \mid f(a,b) = f(b,a)\}, \\
        \sC^2(M, \cA) & = \{h: M\times M\times M \to \cA \text{ with } h(a,b,c)\in \cA(abc) \mid h(a,b,c) - h(b,a,c) + h(b,c,a) = 0 \}
      \end{align*}
      Define coboundary maps:
      \begin{align*}
        (\partial^0 g)(a, b)    & = a_* g(b) - g(ab) + b_* g(a), \\
        (\partial^1 f)(a, b, c) & = a_* f(b,c) - f(ab, c) + f(a, bc) - c_* f(a,b).
      \end{align*}
    \end{De}

    \begin{Pro}
      We have \(\partial^1 \circ \partial^0 = 0\) and thus
      \begin{eqnarray} \label{eqn:cochain}
        \sC^0(M, \cA) \xto{\partial^0} \sC^1(M, \cA) \xto{\partial^1} \sC^2(M, \cA)
      \end{eqnarray}
      is a cochain complex.
    \end{Pro}

    \begin{proof}
      Compute \(\partial^1(\partial^0 g)(a,b,c)\):
      \begin{align*}
         & = a_*(\partial^0 g)(b,c) - (\partial^0 g)(ab, c) + (\partial^0 g)(a,bc) - c_*(\partial^0 g)(a,b) \\
         & = a_*(b_* g(c) - g(bc) + c_* g(b))                                                               \\
         & \quad - ((ab)_* g(c) - g(abc) + c_* g(ab))                                                       \\
         & \quad + (a_* g(bc) - g(abc) + (bc)_* g(a))                                                       \\
         & \quad - c_*(a_* g(b) - g(ab) + b_* g(a)).
      \end{align*}
      Using \((ab)_* = a_* \circ b_*\), \((bc)_* = b_* \circ c_*\), and \(a_*(c_* g(b))
      = (ac)_*(g(b)) = c_*(a_* g(b))\) (commutativity), all terms cancel in pairs.
    \end{proof}

    The cohomology groups of the cochain complex~\eqref{eqn:cochain} are denoted as:
    \begin{align*}
      \sD^0(M, \cA) & = \ker \partial^0 = \{g \mid g(ab) = a_* g(b) + b_* g(a)\;\forall a,b\}, \\
      \sD^1(M, \cA) & = \ker \partial^1 / \operatorname{im} \partial^0.
    \end{align*}

    It is immediate that elements of \(\sD^0(M, \cA)\) are exactly the
    \emph{derivations from \(M\) with values in \(\cA\)}, see
    Definition~\ref{def:derivation}.

\subsection{Long exact sequences}\label{sec:long_exact}

    We say that
    \[ 0 \to \cA' \xto{i} \cA \xto{p} \cA'' \to 0 \]
    is a short exact sequence of systems of abelian groups over \(M\) if it is an exact
    sequence of abelian groups at each fibre. In this case, it induces a long exact
    sequence in the Grillet cohomology, as we show below.

    \begin{Pro} \label{prop:les_grillet}
      Let \(0 \to \cA' \xto{i} \cA \xto{p} \cA'' \to 0\) be a short exact sequence in
      \(\cH(M)\). There is a natural long exact sequence
      \begin{align*}
        0 \to \sD^0(M,\cA') \xto{i_*} \sD^0(M,\cA) \xto{p_*}
        \sD^0(M,\cA'') \xto{\delta^0} \sD^1(M,\cA') \xto{i_*}
        \sD^1(M,\cA) \xto{p_*} \sD^1(M,\cA'').
      \end{align*}
    \end{Pro}

    \begin{proof}
      This is a fairly obvious fact. For each fibre, we have a short exact sequence
      \[ 0 \to \cA(m)' \xto{i} \cA(m) \xto{p} \cA(m)'' \to 0 \]
      and for each \((m,k): m\to mk\), we have maps
      \[
        \xymatrix{
          0\ar[r] & \cA(m)' \ar[r]^{i}\ar[d]_{k_*} & \cA(m) \ar[r]^{p} \ar[d]_{k_*} & \cA(m)'' \ar[r] \ar[d]_{k_*} & 0 \\
          0\ar[r] & \cA(mk)' \ar[r]_{i}            & \cA(mk) \ar[r]_{p}             & \cA(mk)'' \ar[r]             & 0.
        }
      \]
      It follows that there exists a short exact sequence of cochain complexes
      \[ 0\to \sC^0(M, \cA')\to \sC^0(M, \cA)\to \sC^0(M, \cA'')\to 0 \]
      and the result follows.
    \end{proof}

\subsection{Classification theorem}

    The core arguments of the following can implicitly already be found in \cite{gr}.

    Recall the following construction: For a homomorphism \(\alpha: A\to B\) of abelian
    groups, consider the category (groupoid) whose objects are elements of \(B\) and
    whose morphisms \(b' \to b\) are elements \(a\in A\) such that \(\alpha(a) + b =
    b'\). The composition law is induced the group law. Indeed, this groupoid is always
    a symmetric categorical group, with the operation being the group laws. We denote
    this by \([A\xto{\alpha}B]\).

    \begin{Th} \label{thm:classification}
      Let \(M\) be a commutative monoid and \(\cA \in \cH(M)\). There is an equivalence
      of groupoids
      \[ \coex(M, \cA) \;\simeq\; [\sC^0(M, \cA) \xto{\partial^0} \ker\partial^1], \]
      where the right-hand side is the groupoid obtained from the action of \(\sC^0(M,
      \cA)\) on \(\ker\partial^1\) via \(\partial^0\).
    \end{Th}

    \begin{proof}
      We construct a functor
      \[ \gamma: [\sC^0(M, \cA) \xto{\partial^0} \ker\partial^1] \to \coex(M, \cA) \]
      and show that it is an equivalence of categories.

      \textbf{Step 1:} To define \(\gamma \) on objects, we take a function
      \(f\in\ker\partial^1\) and consider the set
      \[ _fM \;:=\; \{(x, a) \mid a \in M,\; x \in \cA(a)\}. \]
      We start by showing that defining
      \[ (x, a)(y, b) \;=\; (b_*(x) + a_*(y) + f(a,b),\; ab) \]
      as our multiplication endows \(_fM\) with the structure of a commutative monoid:
      The operation is clearly commutative since \(f(a,b) = f(b,a)\), as \(f \in
      \ker\partial^1 \subseteq \sC^1\) (and this holds in \(\sC^1\)). To see that its
      associative, we start by computing both side. We have
      \begin{align*}
        [(x,a)(y,b)](z,c) & = (b_*(x)+a_*(y)+f(a,b),\,ab)(z,c)                                   \\
                          & = \bigl(c_*(b_*(x)+a_*(y)+f(a,b)) + (ab)_*(z) + f(ab,c),\; abc\bigr) \\
                          & = [(bc)_*(x) + (ac)_*(y) + c_*(f(a,b)) + (ab)_*(z) + f(ab,c),\; abc]
      \end{align*}
      and
      \begin{align*}
        (x,a)[(y,b)(z,c)] & = (x,a)\bigl(c_*(y)+b_*(z)+f(b,c),\,bc\bigr)                         \\
                          & = \bigl((bc)_*(x) + a_*(c_*(y)+b_*(z)+f(b,c)) + f(a,bc),\; abc\bigr) \\
                          & = [(bc)_*(x) + (ac)_*(y) + (ab)_*(z) + a_*(f(b,c)) + f(a,bc),\; abc].
      \end{align*}
      We see that these agree if and only if \(c_*(f(a,b)) + f(ab,c) = a_*(f(b,c)) +
      f(a,bc)\). That is to say, exactly when \((\partial^1 f)(a,b,c) = 0\), but this
      precisely means that \(f \in \ker \partial^1\).

      Lastly, we focus on the identity. We claim that its \((e_0, 1_M)\), where \(e_0 =
      -f(1_M, 1_M) \in \cA(1_M)\). To verify this, observe that if we specialise
      \(\partial^1 f = 0\) at \((1_M, 1_M, a)\), we get
      \[
        0 \;=\; 1_*(f(1_M,a)) - f(1_M, a) + f(1_M, a) - a_*(f(1_M,1_M))
        \;=\; f(1_M,a) - a_*(f(1_M,1_M)).
      \]
      This gives us \(f(1_M, a) = a_*(f(1_M, 1_M))\) for all \(a\). Now, using \(1_* =
      \id\), we get
      \begin{eqnarray*}
        (e_0, 1_M)(x, a) & = & \bigl(a_*(e_0) + 1_*(x) + f(1_M,a),\, a\bigr)           \\
                         & = & \bigl(-a_*(f(1_M,1_M)) + x + a_*(f(1_M,1_M)),\, a\bigr) \\
                         & = & (x, a).
      \end{eqnarray*}
      Having shown that \(_fM\) is a commutative monoid, we will now show that it is
      indeed a coextension. Defining \(\pi: {_fM} \to M\), \((x,a) \mapsto a\),
      surjectivity is a given. We can define an action of \(z \in \cA(m)\) on the fibre
      \(\pi^{-1}(m)\) by \(z \pt (x, m) := (z + x, m)\). This is clearly regular, since
      for a given \((x,m)\) and \((x',m)\), we can set \(z = x' - x\). For the second
      condition, take take \(z \in \cA(m)\), \(w \in \cA(m')\), and elements \((x,m)\),
      \((y,m')\). We have
      \begin{eqnarray*}
        (z \pt (x,m))(w \pt (y,m')) & = & (z+x, m)(w+y, m')                                      \\
                                    & = & (m'_*(z+x) + m_*(w+y) + f(m,m'),\, mm')                \\
                                    & = & (m'_*(z) + m'_*(x) + m_*(w) + m_*(y) + f(m,m'),\, mm') \\
                                    & = & (m'_*(z)+m_*(w)) \pt (x,m)(y,m'),
      \end{eqnarray*}
      which is the compatibility condition in Definition~\ref{def:coext}(ii), and we
      have successfully shown that
      \[ _f\cE \;=\; (0\to \cA\to _fM \xto{\pi} M\to 0) \]
      is a coextension. The functor \( \gamma\) is given by \(\gamma(f)=\, _f\cE\) on
      objects.\\[2pt]

      \textbf{Step 2:} Next, we define \(\gamma\) on morphisms. A morphism \(f'\to f\)
      in \( [\sC^0(M, \cA) \xto{\partial^0} \ker\partial^1]\) is, by definition, given
      through \(g\in \sC^0(M, \cA) \), such that \(f'=f+\partial^1(g)\). We will show
      that the map \(\theta:\, _{f'}M\to \, _fM\) is a monoid homomorphism, where
      \[ \theta(x,a) \;=\; (x+g(a), a), \ \ \ a\in M, x\in \cA(a). \]
      Indeed, we have
      \[ \theta(x,a)\theta(y,b) \;=\; (b_*x+b_*g(a)+a_*y+a_*g(b)+f(a,b),ab) \]
      and
      \[
        \theta((x,a)(y,b)) \;=\; \theta(b_*x+a_*y+f'(a,b),ab) \;=\; (b_*x+a_*y+f'(a,b)
        +g(ab), ab).
      \]
      These expressions are the same, since \(f' = f + \partial^0(g)\), proving
      \(\theta\) is a homomorphism. Since the diagram
      \[
        \xymatrix{
          0 \ar[r] & \cA \ar[r] \ar[d]_{\id} & _{f'}M \ar[r]^{\pi'} \ar[d]^{\theta} & M \ar[r] \ar[d]^{\id} & 0 \\
          0\ar[r]  & \cA \ar[r]              & _fM \ar[r]_{\pi}                     & M \ar[r]              & 0.
        }
      \]
      commutes, we obtain a morphism of coextensions \((\id,\theta,\id):\,_{f'}\cE\to
      \, _f\cE\). This allows us to define the functor \(\theta\) on morphisms by
      \(\theta(g)=(\id,\theta,\id)\).

      \textbf{Step 3:} We proceed to showing that \(\gamma \) is full and faithful. Let
      \((\id, \zeta, \id):\,_{f'}\cE\to \, _f\cE\) be a morphism of coextensions where
      \(\zeta: \,_{f'}M\to \,_{f}M\) is a monoid homomorphism compatible with the
      actions of \(\cA\). By commutativity of the diagram, it follows that \(\zeta(0,
      a) = (g(a), a)\) for some \(g(a)\in \cA(a)\). Hence \(g\in \sC^0(M, \cA)\). It
      follows that
      \[
        \zeta(x, a) \;=\; \zeta (x\pt (0, a) \;=\; x\pt \zeta(0, a) \;=\; x\pt (g(a),
        a) \;=\; (x + g(a), a)
      \]
      form the compatibility of the action. The fact that \(\zeta\) is a homomorphism
      of monoids implies that \(g\) defines a homomorphism \(f'\to f\) in \([\sC^0(M,
      \cA) \xto{\partial^0} \ker\partial^1]\). Thus, \(\gamma\) is a bijection from the
      set \(\Hom(f', f)\) to \(\Hom_{\coex(M,\cA)}(_{f'}\cE, _{f}\cE)\) and Step~3 is
      proved.

      \textbf{Step 4:} To prove our equivalence, it remains to show that the functor
      \(\gamma\) is essentially surjective. Let
      \[ 0 \to \cA \to N \xto{\pi} M \to 0 \in \coex(M,\cA), \]
      be a coextension. We wish to show that there exists an \(f\in \ker\partial^1\)
      such that \(N\simeq _fM\) is an isomorphism of coextensions. To this end, choose
      lifts \(\sigma(a) \in \pi^{-1}(a)\) for each \(a \in M\).

      For \(a, b \in M\), both \(\sigma(ab)\) and \(\sigma(a)\sigma(b)\) lie in
      \(\pi^{-1}(ab)\) and so, by regularity, there exists a unique \(f(a,b) \in
      \cA(ab)\) with
      \[ \sigma(a)\sigma(b) \;=\; f(a,b) \pt \sigma(ab). \]
      This defines a function \(f(a,b)\) in two variables wich will satisfy \(f(a,b) =
      f(b,a)\). Thus, it is not hard to see that we have defined \(f \in
      \sC^1(M,\cA)\).

      To check \(f \in \ker \partial^1\), we essentially have to consider the
      associativity of computing \(\sigma(a) \sigma(b) \sigma(c)\). On the one hand, we
      have
      \begin{eqnarray*}
        (\sigma(a)\sigma(b))\sigma(c) & = & (f(a,b) \pt \sigma(ab))\sigma(c) \\
                                      & = & (c_*(f(a,b)) + f(ab,c)) \pt \sigma(abc).
      \end{eqnarray*}
      On the other,
      \[
        \sigma(a)(\sigma(b)\sigma(c)) \;=\; (a_*(f(b,c)) + f(a,bc)) \pt \sigma(abc).
      \]
      Equating these gives \((\partial^1 f)(a,b,c) = 0\).

      Recall the definition
      \[ _fM \;:=\; \{(x, a) \mid a \in M,\; x \in \cA(a)\}. \]
      Hence, in order to define \(N\to {_fM}\), given a \(\pi(a)\in M\), we have to
      find an \(x\in \cA(a)\) that will extend to an isomorphism. This is imply
      regularity, since \(\pi(a) = \pi(\sigma(\pi(a)))\) and thus, there exists a
      unique \(x_a\in\cA(\pi(a))\) such that \(a = x_a \pt \sigma(\pi(a))\). Defining
      \[ \varphi: N \to {_fM} \,\qtext{by}\, \varphi(a) \;=\; (x_a, \pi(a)) \]
      gives us our isomorphism of coextensions.

      \textbf{Step 5:} Lastly, it remains to verify that the equivalence \(\gamma\)
      respects the symmetric categorical group structures. In other words, that the
      Baer sum corresponds to the natural addion of
      \([\sC^0(M,\cA)\xto{\partial^0}\ker\partial^1]\).

      Let \(f, g \in \ker\partial^1\). We must produce an isomorphism of coextensions
      \[
        \gamma(f+g) \;=\; {}_{f+g}\cE
        \;\xrightarrow{\;\sim\;}
        \gamma(f)+\gamma(g) \;=\; {}_{f}\cE + {}_{g}\cE,
      \]
      where the right-hand side is the Baer sum
      \(\nabla_*\!\left(\Delta^*({}_{f}\cE\times{}_{g}\cE)\right)\).

      Recall that the underlying monoid of \(\Delta^*({}_{f}\cE\times{}_{g}\cE)\) is
      \[
        T \;:=\;
        \bigl\{\,(x,y,a)\;\bigm|\;a\in M,\;x\in\cA(a),\;y\in\cA(a)\,\bigr\}
      \]
      with componentwise multiplication
      \[
        (x,y,a)\cdot(x',y',a')
        \;=\;
        \bigl(b'_*(x)+a_*(x')+f(a,a'),\; a'_*(y)+a_*(y')+g(a,a'),\; aa'\bigr).
      \]
      Pushing forward along \(\nabla:\cA\oplus\cA\to\cA\), \((u,v)\mapsto u+v\),
      identifies the \(\cA\oplus\cA\)-orbit of \((x,y,a)\) with the element
      \([x+y,\,a]\in{}_{f+g}M\), via the map
      \[
        \Phi: T \;\longrightarrow\; {}_{f+g}M,
        \qquad
        \Phi\bigl(x,y,a\bigr) \;=\; (x+y,\,a).
      \]
      We verify that \(\Phi\) is a monoid homomorphism:
      \begin{eqnarray*}
        \Phi\bigl((x,y,a)\cdot(x',y',a')\bigr) & = & \Phi\bigl(a'_*(x)+a_*(x')+f(a,a'),\; a'_*(y)+a_*(y')+g(a,a'),\; aa'\bigr) \\
                                               & = & \bigl(a'_*(x)+a_*(x')+f(a,a')+a'_*(y)+a_*(y')+g(a,a'),\;aa'\bigr)         \\
                                               & = & \bigl(a'_*(x+y)+a_*(x'+y')+(f+g)(a,a'),\;aa'\bigr)                        \\
                                               & = & \Phi\bigl(x,y,a\bigr)\cdot\Phi\bigl(x',y',a'\bigr).
      \end{eqnarray*}
      Moreover, \(\Phi\) respects the \(\cA\)-actions as for \(z\in\cA(a)\), we have
      \[
        \Phi\bigl(z\pt\bigl(x,y,a\bigr)\bigr)
        \;=\; \Phi\bigl(z+x,y,a\bigr)
        \;=\; (z+x+y,a)
        \;=\; z\pt(x+y,a)
        \;=\; z\pt\Phi\bigl(x,y,a\bigr).\qedhere
      \]
    \end{proof}

    We have the following results in the notation of Theorem~\ref{thm:classification}:
    \begin{Cor} \label{cor:classification_corrolaries}
      There are isomorphisms
      \[ \pi_0(\coex(M, \cA))\simeq \sD^1(M, \cA) \]
      and
      \[ \pi_1(\coex(M, \cA))\simeq \sD^0(M, \cA), \]
      where we recall that \(\coex(M, \cA)\) is a categorical group and hence, the
      automorphism of every object is isomorphic.

      \noindent Moreover, the split coextension (semi-direct product) corresponds to
      \(f = 0\).
    \end{Cor}

    \begin{Rem}[Free monoids]
      As an immediate application of Corollary~\ref{cor:classification_corrolaries}, we
      obtain that if \(M\) is a free monoid, then \(\sD^1(M, -) = 0\). In particular,
      any coextension splits if \(M\) is free.

      \noindent We can also deduce form
      Proposition~\ref{prop:monoig_coex_generalises_group_coex} that if \(M = G\) is a
      group, then
      \[ \sD^1(G, \cA)=Ext(G,\cA(1)). \]
    \end{Rem}

\subsection{An exact sequence}

    Let \(\cE = (0 \to \cB \to K \xto{q} M \to 0)\) be a coextension and \(\cA \in
    \cH(M)\) a system of abelian groups. Recall that for any morphism \(\alpha: \cB \to
    \cA\) in \(\cH(M)\), the push-forward construction gives a coextension
    \(\alpha_*(\cE) \in \coex(M, \cA)\), see Theorem~\ref{thm:classification}. This
    yields a group homomorphism
    \[
      \tau: \Hom_{\cH(M)}(\cB, \cA) \to \sD^1(M, \cA),
      \qquad \tau(\alpha) \;=\; [\alpha_*(\cE)].
    \]

    \begin{De}
      Let \(0 \to \cA \to N \xto{\pi} M \to 0\) be a coextension. Define \(D(\pi)\) to
      be the subgroup of all derivations \(\partial: N \to \pi^*(\cA)\) satisfying
      \begin{equation}
        \label{eq:deriv_cond}
        \partial((x + y) \pt n) + \partial(n)
        \;=\; \partial(x \pt n) + \partial(y \pt n),
      \end{equation}
      where \(x, y \in \cA(m)\) and \(n \in N\) with \(\pi(n) = m\).
    \end{De}

    \begin{Pro} \label{prop:morphism_exact}
      Let \(0 \to \cB \to K \xto{q} M \to 0\) be a coextension and \(\cA \in \cH(M)\).
      There exists an exact sequence
      \[
        0 \to \sD^0(M, \cA) \xto{\eta} D(q) \xto{\theta} \Hom_{\cH(M)}(\cB, \cA)
        \xto{\tau} \sD^1(M, \cA).
      \]
    \end{Pro}

    \begin{proof}
      We start by defining the maps. Define
      \[
        \eta: \sD^0(M, \cA) \to D(q), \qquad
        \eta(\delta) \;=\; \partial \;:=\; \delta \circ q,
      \]
      where \(\delta \in \sD^0(M, \cA)\). Condition~\eqref{eq:deriv_cond} is satisfied
      by Lemma~\ref{lem:exact_at_M} and thus, \(\partial \in D(q)\).

      To define \(\theta\), take \(\partial \in D(q)\). We need to construct a natural
      transformation \(\theta(\partial): \cB \to \cA\) in \(\cH(M)\). Let \(m \in M\),
      \(x \in \cB(m)\) and \(n \in q^{-1}(m)\) be a lift. We define
      \[
        \theta(\partial)(m)(x) \;:=\; \partial(x \pt n) - \partial(n) \;\in\; \cA(m),
      \]
      giving us a morphism (read, natural transformation) \(\cB\to \cA\), as desired.
      To see that this is independent of the choice of \(n\), we use regularity and
      observe that for any \(n' \in \pi^{-1}(m)\), we have \(n' = y \pt n\) for a
      unique \(y \in \cB(m)\). Thus \(x \pt n' = x \pt (y \pt n) = (x + y) \pt n\)
      holds, and Condition~\eqref{eq:deriv_cond} gives
      \[
        \partial(x \pt n') - \partial(n') \;=\; \partial((x+y) \pt n) - \partial(y \pt
        n) \;=\; \partial(x \pt n) - \partial(n).
      \]
      The last equation is a mirror of the first. Checking that \(\theta(\partial)(m)\)
      are compatible group homomorphisms is left to the reader, as it is direct
      verification.

      \medskip We now wish to verify exactness. We start by checking that {\itshape
      \(\eta\) is injective}. This follows directly from \(q: K \to M\) being a
      surjection, as in this case, \(\delta \circ q = 0\) implies \(\delta = 0\).

      To show that \(\im(\eta) = \ker(\theta)\) (meaning \textit{exactness at
      \(D(q)\)}), take \(\partial = \eta(\delta) = \delta \circ q\). We have
      \(\theta(\partial)(m)(x) = \partial(x \pt k) - \partial(k) = \delta(q(x \pt k)) -
      \delta(q(k)) = 0\), since \(q(x \pt k) = q(k)\), and thus, \(\theta \circ \eta =
      0\), meaning, we land in the kernel.

      Conversely, let \(\partial \in D(q)\) with \(\theta(\partial) = 0\). Then
      \(\partial(x \pt k) = \partial(k)\) for all \(k \in K\) and \(x \in \cB(q(k))\).
      Since \(\cB\) acts freely and transitively on each fibre \(q^{-1}(m)\)
      (regularity), \(\partial\) is constant on each fibre. Hence \(\partial\) factors
      as \(\delta \circ q\) for a uniquely determined map \(\delta: M \to \cA\). One
      checks that \(\delta\) is a derivation (using the surjectivity of \(q\)), and
      thus, \(\delta \in \sD^0(M, \cA)\). It follows that \(\partial = \eta(\delta)\).

      Next, we wish to show that \(\im(\theta) = \ker(\tau)\), meaning {\itshape
      exactness at \(\Hom_{\cH(M)}(\cB, \cA)\)}. Let \(\alpha = \theta(\partial)\) and
      consider the push-forward \(\alpha_*(\cE)\). The derivation \(\partial: K \to
      q^*\cA\) satisfies \(\partial(x \pt k) = \alpha(x) + \partial(k)\), \(x \in
      \cB(q(k))\), which is precisely the condition of Lemma~\ref{lem:split}~(3) for
      \(\alpha_*(\cE)\). As such, we can deduce that \(\alpha\) splits and by
      Corollary~\ref{cor:classification_corrolaries}, we land in the kernel.

      Conversely, let \(\alpha: \cB \to \cA\) satisfy \(\tau(\alpha) = 0\). In other
      words, \(\alpha_*(\cE)\) splits. By Lemma~\ref{lem:split}~(3), there exists a
      derivation \(\partial': K \to q^*\cA\) with \(\partial'(x \pt k) = \alpha(x) +
      \partial'(k)\) for all \(x \in \cB(q(k))\). In particular, this forces
      Condition~\eqref{eq:deriv_cond} for \(\partial'\). To see the last statement,
      observe that
      \[
        \partial'((x+y)\pt k) \;=\; \alpha(x+y) + \partial'(k)
        \;=\; \alpha(x) + \alpha(y) + \partial'(k)
      \]
      and
      \[
        \partial'(x \pt k) + \partial'(y \pt k) - \partial'(k)
        \;=\; \alpha(x) + \partial'(k) + \alpha(y) + \partial'(k) - \partial'(k),
      \]
      and as such, they agree. It follows that \(\partial' \in D(q)\) and it can now be
      verified that \(\theta(\partial') = \alpha\).
    \end{proof}

\subsection{Versal coextensions} \label{sec:versal}

    \begin{De}
      A coextension \(0 \to \cB \to K \xto{q} M \to 0\) is called \emph{versal} if
      there exists a morphism (in \(\coex(M)\)) from the versal one to any coextension
      \(0 \to \cA \to N \xto{\pi} M \to 0\).
    \end{De}

    \begin{Const}
      Choose a surjective homomorphism \(\pi: F \twoheadrightarrow M\) from a free
      commutative monoid \(F\).

      The idea is to build a ``(\sout{uni})versal'' coextension that encodes all
      possible discrepancies between lifts of elements of \(M\) to \(F\). Note that
      it's called versal since it is not unique. In particular, there can be more than
      one morphism form the versal one to a given coextension.

      \medskip\noindent\textbf{The coefficient system \(\cB\).} Let \(\cB = \cB^\pi\)
      be the system of abelian groups over \(M\), defined as follows. As a set,
      \(\cB(m)\) is the free abelian group generated by symbols \(\delta_{f,g}\), one
      for each pair \(f, g \in F\) with \(\pi(f) = \pi(g) = m\). On these, we put the
      following relations:
      \[
        \begin{array}{rcll}
          \delta_{f,f}        & = & 0                                                  & \,\text{for all}\, f \in F, \\
          \delta_{ff_1, gg_1} & = & \pi(f_1)_*\delta_{f,g} + \pi(f)_*\delta_{f_1, g_1} & \,\text{for }\, \pi(f) \;=\; \pi(g),\; \pi(f_1) \;=\; \pi(g_1).
        \end{array}
      \]

      The restriction map \(a_*: \cB(m) \to \cB(am)\) sends \(\delta_{f,g}\) to
      \(\delta_{f f_0, g f_0}\), where \(f_0 \in F\) is any lift of \(a\). This is well
      defined by the second relation.

      \medskip\noindent\textbf{The monoid \(K\).} Define \(K\) to be the set of
      equivalence classes of pairs \((\omega, f)\), where \(f \in F\) and \(\omega \in
      \cB(\pi(f))\). We put an equivalence relation on this where \((\omega, f) \sim
      (\tau, g)\) if and only if
      \[ \pi(f) \;=\; \pi(g) \,\qtext{and}\, \omega - \tau \;=\; \delta_{g, f}. \]
      Write \([\omega, f]\) for the class of \((\omega, f)\). We endow it with a monoid
      structure by declaring
      \[
        [\omega, f] \cdot [\omega_1, f_1]
        \;:=\; [\pi(f_1)_*\omega + \pi(f)_*\omega_1,\; ff_1].
      \]
      \medskip\noindent\textbf{The ``maps''}
      To finish our construction of the coextension, we still need to define the map
      \(K\to M\) and the actions of \(\cB\). We define the surjection/projection \(q: K
      \to M\) by \(q[\omega, f] = \pi(f)\). The action of \(u \in \cB(m)\) on the fibre
      \(q^{-1}(m)\) is given by \(u \pt [\omega, f] = [u + \omega, f]\).
    \end{Const}

    \begin{Le}
      The construction above yields a versal coextension \(0 \to \cB \to K \xto{q} M
      \to 0\).
    \end{Le}

    \begin{proof}
      It is not hard to verify that this is a coextension and is left to the reader. We
      only verify that it is versal. For this, let \(0 \to \cA \to N \xto{\pi} M \to
      0\) be a coextension.

      Since \(F\) is free, there is a monoid homomorphism \(\phi: F \to N\) with \(\pi
      \circ \phi = q\). Hence, for \(f, g \in F\) with \(p(f) = m = p(g)\), we have
      \(\pi(\phi(f)) = \pi(\phi(g)) = m\). Regularity now tells us that there is a
      unique \(x_{f,g} \in \cA(m)\) with \(\phi(f) = x_{f,g} \pt \phi(g)\). It can be
      readily seen that \(x_{f,g}\) respects the two conditions we imposed on
      \(\delta_{f,g}\) and thus, the assignment \(\delta_{f,g} \mapsto x_{f,g}\) is
      well defined, giving us a morphism \(\alpha: \cB \to \cA\) in \(\cH(M)\). The map
      \(\psi: K \to N\), given by \(\psi([\omega, f]) = \alpha(\omega) \pt \phi(f)\),
      is now gives us the required morphism of coextensions
      \[
        (\alpha, \psi, \id_M): (0 \to \cB \to K \to M \to 0)
        \to (0 \to \cA \to N \to M \to 0).\qedhere
      \]
    \end{proof}

    We can sharper Proposition~\ref{prop:morphism_exact} for versal coextensions.

    \begin{Cor}[Versal exact sequence]
      Let \(0 \to \cB \to K \xto{q} M \to 0\) be a versal coextension and \(\cA \in
      \cH(M)\). There is an exact sequence
      \[
        0 \to \sD^0(M, \cA) \xto{\eta} D(q) \xto{\theta} \Hom_{\cH(M)}(\cB, \cA)
        \xto{\tau} \sD^1(M, \cA) \to 0.
      \]
    \end{Cor}

    \begin{Rem}[Naturality of the versal exact sequence]
      Let \(\cE_0 = (0 \to \cB \to K \to M \to 0)\) be a versal coextension. For any
      coextension \(\cE = (0 \to \cA \to N \to M \to 0)\), the versality provides a
      morphism \((\alpha, \psi, \id_M): \cE_0 \to \cE\). By
      Proposition~\ref{prop:les_grillet}, \(\alpha\) induces a commutative ladder
      \[
        \xymatrix{
          0 \ar[r]                 & \sD^0(M,\cB) \ar[r] \ar[d]^{\alpha_*}               & D(q) \ar[r] \ar[d]
                                   & \Hom(\cB,\cB) \ar[r]^{\tau} \ar[d]^{\alpha \circ -} & \sD^1(M,\cB)
          \ar[r] \ar[d]^{\alpha_*} & 0 \\
          0 \ar[r]                 & \sD^0(M,\cA) \ar[r]                                 & D(q,\cA) \ar[r]
                                   & \Hom(\cB,\cA) \ar[r]^{\tau}                         & \sD^1(M,\cA)
          \ar[r]                   & 0,
        }
      \]
      in which the bottom row is the versal exact sequence for \(\cA\), and the
      vertical map \(\Hom(\cB,\cB) \to \Hom(\cB,\cA)\) is post-composition with
      \(\alpha\). This ladder is natural in \(\cE\).
    \end{Rem}

\section{Coextensions of Monoid Schemes}\label{sec:global}

\subsection{Sheaves over posets} Let \((P, \leq)\) be a poset, meaning partially
    ordered set. For \(x\in P\), we can consider the subset \(L_x(P) := \{y\in P\mid
    y\leq x\}\). We call \(P\) a \emph{locally lattice poset} if for all \(x\in P\),
    \(L_x(P)\) is a semilattice.

    We can define a topology on a poset, which we call the \(\fP\)-topology, by
    declaring subsets \(U = \{ y\in P \mid x\in U, y\leq x \Rightarrow y\in U\}\) to be
    open. This is equivalent to the topology whose basis consists of all the sets of
    the form \(L_x(P)\). It is easily seen that defining a sheaf on a poset with this
    topology is equivalent to defining a functor on the underlying poset, since the
    sheaf condition falls away: the basis is in bijection with the elements of \(P\),
    and no basic open can be expressed as a non-trivial union of other basic opens.

    \begin{De}
      A sheaf \(\cF: P\to \bC\) over a poset topology (or simply poset) with values in
      a category \(\bC\) is a contravariant functor from the underlying poset to
      \(\bC\). A morphism \(\cF\to \cG\) of sheaves is just a natural transformation.
    \end{De}

    For a sheaf \(\cF\) over a poset \(p\), the value of \(\cF\) on \(p\in P\) is
    denoted by \(\cF_p\) and is called the \emph{stalk} at \(p\). Moreover, if \(q\leq
    p\), the structural morphism \(\cF_p\to \cF_q\) is usually denoted by
    \(\phi_{p,q}\).

\subsection{Monoid schemes}

    We will work with monoid schemes of finite type, following \cite{p1,p2}. Below, we
    give a quick summary, the details of which can be checked in one form or another in
    the above sources.

\subsubsection{Monoid schemes of finite type}

      This brings us to the following definition.

      \begin{De}
        A \emph{monoid scheme of finite type} is a pair \((X, \cO_X)\), where \(X\) is
        a finite locally lattice poset and \(\cO_X: X^{op} \to \FGMon \) is a sheaf
        (meaning covariant functor over its underlying poset) of finitely generated
        monoids over \(X\), such that: For each \(x \in X\), the restriction of
        \(\cO_X\) to \(L_x(X)\) is isomorphic to \(\Spec(\cO_{X,x})\), equipped with
        the localisation functor \(\p \mapsto (\cO_{X,x})_\p\).

        A \emph{morphism} \((\varphi, \eta): (X, \cO_X) \to (Y, \cO_Y)\) is a poset map
        \(\varphi: X \to Y\), and a natural transformation \(\eta: \cO_Y \circ \varphi
        \To \cO_X\) such that each \(\eta_x: \cO_{Y,\varphi(x)} \to \cO_{X,x}\) is a
        local monoid homomorphism.

        For \(y \leq x\) in \(X\), we write \(i_{x,y}: \cO_{X,x} \to \cO_{X,y}\) for
        the structural homomorphism. By the definition of a monoid scheme, this is a
        localisation.
      \end{De}

      \begin{Rem}
        Note that the requirement to ``land'' in finitely generated monoids is likely
        not needed, as the finiteness of \(X\) clearly implies the finiteness of
        \(L_x(X)\). Hence, it forces the structural monoid on \(L_x(X)\) to have finite
        spectrum. This might indeed be enough to make every part of the theory work, as
        the major obstructions come from non-finite spectra. However, we will table
        this discussion and use the simpler definition for now, as the core premiss of
        this paper is going in a different direction.
      \end{Rem}

      We call a monoid scheme \(X\) of finite type \emph{quasi-separated} if the
      intersection of any two open affine subschemes is again affine.

      \begin{Conv}
        Henceforth, unless otherwise stated, a monoid scheme is assumed to be
        quasi-separated (and subsequently, of finite type).
      \end{Conv}

      As already mentioned, the sheaf condition is satisfied trivially.

\subsection{Coextensions of monoid sheaves}

    Our main aim is to consider the coextensions of monoid schemes. Monoid schemes are,
    naturally, sheaves of monoids in the first place. We will subsequently first
    develop the theory of coextensions for monoid sheaves and then specialise for
    monoid schemes. It should be noted, however, that these settings are genuinely
    distinct. Indeed, the main results of this section, which is surprising perhaps
    even from a classical point of view, that coextensions form a stack, only holds for
    monoid schemes and not monoid sheaves.

    Throughout, let \(P\) be a finite poset and \(\cF\) a sheaf of monoids.

    \begin{De} \label{def:system_sheaf}
      A \emph{system of abelian groups over \(\cF\)} is the following datum:
      \begin{itemize}
        \item A system of abelian groups \(\cA_x \in \cH(\cF_x)\) for each \(x \in P\);
        \item A morphism \(\alpha_{x,y}: \cA_x \to \phi_{x,y}^*\cA_y\) in
              \(\cH(\cF_x)\) for each \(y \leq x\),
      \end{itemize}
      satisfying \(\alpha_{x,x} = \id\) and, for \(z \leq y \leq x\),
      \[
        \alpha_{x,z} \;=\; \phi_{x,y}^*(\alpha_{y,z}) \circ \alpha_{x,y}
        \quad \,\text{in }\, \cH(\cF_x).
      \]
      A morphism \(\psi: \cA \to \cB\) in \(\cH(\cF)\) consists of morphisms \(\psi_x:
      \cA_x \to \cB_x\) in \(\cH(\cF_x)\) for each \(x\), compatible with the
      \(\alpha_{x,y}\). The resulting category is denoted by \(\cH(\cF)\).
    \end{De}

    \begin{De} \label{def:coext_sheaf}
      Let \(\cA \in \cH(\cF)\) be a system of abelian groups over \(\cF\). A
      \emph{group coextension} of \(\cF\) over \(P\) by \(\cA\), written \(0 \to \cA
      \to \cG \to \cF \to 0\), is a sheaf of monoids \(\cG\) over \(P\) together, with
      a morphism of sheaves \(\pi: \cG \to \cF\), such that
      \begin{itemize}
        \item for each \(x\in P\), \(0\to \cA_x \to \cG_x \xto{\pi_x} \cF_x \to 0\) is
              a coextension of monoids (Definition~\ref{def:coext}) and;
        \item for each \(y \leq x\) in \(P\), the following
              \[
                \xymatrix{
                  0 \ar[r] & \cA_x \ar[r] \ar[d]_{\alpha_{x,y}} & \cG_x \ar[r]^{\pi_x} \ar[d]_{\phi_{x,y}^{\cG}} & \cF_x \ar[r] \ar[d]_{\phi_{x,y}^{\cF}} & 0 \\
                  0 \ar[r] & \cA_y \ar[r]                       & \cG_y \ar[r]^{\pi_y}                           & \cF_y \ar[r]                           & 0
                }
              \]
              is a morphism of sheaf coextensions
              (Definition~\ref{de:map_coextensions}).
      \end{itemize}
      A \emph{morphism} of coextensions of \(\cF\) by \(\cA\) is a natural
      transformation \(h: \cG \to \cG'\) of sheaves of monoids, such that
      \((\id_{\cA_x}, h_x, \id_{\cF_x})\) is a morphism of coextensions of \(\cF_x\)
      for each \(x\). The resulting category is denoted by \(\coex(\cF, \cA)\).
    \end{De}

    By Proposition~\ref{prop:coex_catgroup}, each \(\coex(\cF_x, \cA_x)\) is a
    symmetric categorical group under the Baer sum. The category \(\coex(\cF, \cA)\)
    inherits this structure, as the Baer sum is defined stalkwise, just like the
    semi-direct product, which plays the unit.

\subsection{Coextensions of monoid schemes}

    Let \(X\) be a quasi-separated monoid scheme of finite type and \(\cA\) a system of
    abelian groups over \(X\). Each structural map \(\phi_{x,y} = i^{x,y}: \cO_{X,x}
    \to \cO_{X,y}\) is a localisation in this setting. Hence, by adjointness, every
    structural morphism \(\alpha_{x,y}:\cA_x\to i^{x,y}\cA_y\) in the
    Definition~\ref{def:system_sheaf} induces the morphism
    \(\alpha^{x,y}:i^{x,y}_!\cA_x\to \cA_y\).

\subsubsection{Systems of abelian groups over monoid schemes}

      \begin{De} \label{def:system_monoid}
        Let \(X\) be a monoid scheme. A \emph{system of abelian groups over \(X\)} is a
        system of abelian groups over the underlying sheaf \(\cO_X\) in which each
        \(\alpha^{x,y}: i^{x,y}_!\cA_x \to \cA_y\) is an isomorphism in
        \(\cH(\cO_{X,x})\). The resulting category is denoted by \(\cH(X)\).

        A \emph{morphism} \(\phi: \cA \to \cB\) of systems over \(X\) consists of
        morphisms \(\phi_x: \cA_x \to \cB_x\) in \(\cH(\cO_{X,x})\) for each \(x\),
        compatible with the isomorphisms \(\alpha_{x,y}\). The resulting category is
        denoted by \(\cH(X)\).
      \end{De}

      \begin{Nota}
        Henceforth when writing \(\Spec(M)\), unless explicitly stated otherwise, we
        will mean the affine scheme of \(M\), not the idempotent monoid of prime ideals
        of \(M\). A major exception is Section~\ref{sec:semilattice}.
      \end{Nota}

      \begin{Le}
        Let \(M\) be a finitely generated monoid. There is an equivalence \(\cH(M)
        \simeq \cH(\Spec(M))\).
      \end{Le}

      \begin{proof}
        The functor \(\cH(M) \to \cH(\Spec(M))\) sends \(\cA\) to the system \(\p
        \mapsto \cA_\p\), with transition maps given by further localisation (which are
        isomorphisms). The inverse takes the stalk at the unique maximal ideal \(\fm =
        M \setminus M^\times\), which recovers \(\cA\) by the
        definition~\ref{def:system_monoid}.
      \end{proof}

\subsubsection{Coextensions over monoid schemes}

      \begin{Pro}
        Let \(\cA \in \cH(X)\) be a system of abelian groups over a monoid scheme \(X\)
        and
        \[
          0 \to \cA \to \cG \to {\mathcal O}_X
          \to 0
        \]
        a \emph{group coextension} of the structure sheaf \(\cO_X\). There exist a
        unique monoid scheme \(Y\) whose global section \(\cO_Y\) agrees with \(\cG\),
        and whose underlying poset is the same as that of \(X\).

        Moreover, there is a unique morphism of sheaves \(Y\to X\) which is the
        identity on the underlying posets, and given by \(\cG \to \cO_x\) as structure
        sheaves.
      \end{Pro}

      \begin{proof}
        This statement is local in nature and so, we can assume \(X = \Spec(M)\). Let
        \(\fm\) be the maximal ideal of \(M\). This gives us a group coextension of the
        monoid \(M = M_\fm\) (they agree with each other since \(\fm = M\setminus
        M^\times\))
        \[ 0\to \cA_\fm \to \cG_\fm \xto{\pi_\fm} M\to 0. \]
        Set \(N = \cG_\fm\). By Corollary~\ref{cor:spec_coext}, we can identify the
        underlying posets of \(\Spec(M)\) and \(\Spec(N)\), and use the same notations
        for prime ideals of \(M\) and \(N\). Take any prime ideal \(\p\) of \(M\) and
        consider the commutative diagram
        \[
          \xymatrix{
            0 \ar[r] & \cA_{\sf m} \ar[r] \ar[d] & N \ar[r]^{\pi_{\sf m}} \ar[d]^\phi & M \ar[r] \ar[d] & 0 \\
            0 \ar[r] & \cA_\p \ar[r]             & \cG_p \ar[r]^{\pi_\p}              & M_\p \ar[r]     & 0
          }
        \]
        from Definition~\ref{def:coext_sheaf}. Denote by \(S = M\setminus \p\) and \(T
        = \pi_{\sf m}^{-1}(S)\) and take an element \(t\in T \subset N\). Since the
        image of \(t\) in \(M_\p\) is invertible, we can use Lemma~\ref{lem:inv_ext} to
        conclude that \(\phi(t)\) is also invertible. It follows that \(\phi\) induces
        the homomorphism \(\psi: N_\p\to \cG_\p\), which fits in the diagram
        \[
          \xymatrix{
            0 \ar[r] & \cA_{\p} \ar[r] \ar[d] & N_\p \ar[r] \ar[d]^\psi & M_\p \ar[r] \ar[d] & 0 \\
            0 \ar[r] & \cA_\p \ar[r]          & \cG_p \ar[r]^{\pi_\p}   & M_\p \ar[r]        & 0,
          }
        \]
        thanks to Lemma~\ref{lem:localisation_coext}. We can now use
        Lemma~\ref{lem:short5} to conclude that \(\psi\) is an isomorphism. Thus \((X,
        \cG)\) is isomorphic to \(\Spec(N)\) and proof is finished.
      \end{proof}

      \begin{Rem}
        Note that it is not required for the \(\pi_x\)'s to be a local sheaf
        homomorphism, nor for the morphisms \(Y_x \to Y'_x\) between coextensions in
        \(\coex(X,\cA)\). All this follows automatically. However, if we were to
        consider morphisms between coextensions over different bases, the locality
        between the monoid schemes would need to be imposed explicitly.
      \end{Rem}

\subsubsection{The restriction functor and the 2-functor structure}

      The major difference between monoid sheaves and monoid schemes is that we have
      functors between coextensions
      \[ \coex(\cO_{X,x}, \cA_x) \to \coex(\cO_{X,y}, \cA_y) \]
      for each \(y \leq x\), induced by localisations. As such, the coextensions of a
      monoid scheme is no longer just a single symmetric categorical group, but instead
      a stack, as we will argue. This has many pleasant consequences as might be
      expected, and we will try to outline a few of these in our paper.

      \begin{De}
        Let \(y \leq x\) in \(X\). We define the \emph{restriction functor}
        \[
          \rho_{x,y}: \coex(\cO_{X,x}, \cA_x)
          \;\rightarrow\; \coex(\cO_{X,y},
          \cA_y)
        \]
        as follows: For \(\cE_x = (0 \to \cA_x \to \cO_{Y,x} \xto{\pi_x} \cO_{X,x} \to
        0) \in \coex(\cO_{X,x}, \cA_x)\), Lemma~\ref{lem:localisation_coext} applied to
        the localisation \(i^{x,y}: \cO_{X,x} \to \cO_{X,y}\), yields the coextension
        \[
          \rho_{x,y}(\cE_x)
          \;:=\; 0 \to i^{x,y}_!(\cA_x)\simeq \cA_y \to \cO_{Y,y} \to \cO_{X,y} \to 0.
        \]
      \end{De}

      To elaborate on the isomorphism of the first term, \(i^{x,y}_!(\cA_x) =
      S^{-1}\cA_x\) is the localisation of \(\cA_x\). Since \(\alpha_{x,y}: \cA_x
      \xto{\sim} (i^{x,y})^*\cA_y\) is an isomorphism in \(\cH(\cO_{X,x})\), it
      corresponds, under the adjunction \(i^{x,y}_! \dashv (i^{x,y})^*\), to an
      isomorphism \(\bar\alpha_{x,y}: i^{x,y}_!(\cA_x) \xto{\sim} \cA_y\) in
      \(\cH(\cO_{X,y})\).

      \begin{Le} \label{lem:rho_2_functor}
        The functors \(\rho_{x,y}\) make \(x \mapsto \coex(\cO_{X,x}, \cA_x)\) into a
        2-functor (pseudofunctor) from \(X^{op}\) to the 2-category of symmetric
        categorical groups.
      \end{Le}

      \begin{proof}
        To show that its a 2-functor, we have to verify that each \(\rho_{x,y}\) is a
        morphism of symmetric categorical groups and that for \(z \leq y \leq x\),
        there is a natural isomorphism \(\rho_{y,z} \circ \rho_{x,y} \simeq
        \rho_{x,z}\).

        The first assertion follows directly from the fact that localisation commutes
        with fibre products and with the fold map \(\nabla\).

        The second assertion follows from the fact that the isomorphisms
        \(\bar\alpha_{x,y}\) and \(\bar\alpha_{y,z}\) compose to \(\bar\alpha_{x,z}\)
        by the cocycle condition on \(\alpha_{x,y}\).
      \end{proof}

\subsubsection{The stack property and global sections}

      We now come to one of the main results of this paper, which is
      Theorem~\ref{thm:global-sections-2limit}. For this, recall that since \(X\) is a
      finite poset endowed with the poset (or \(\fP\))-topology, any pseudofunctor on
      \(X^{op}\) automatically defines a unique stack on the poset topology. This is
      done by taking the 2-limit over the Ĉech complex associated to the cover of open
      sets by principal open sets (which correspond exactly to the points of the
      underlying poset). In other words, for any \(U\subseteq P\), we take the 2-limit
      \(\tl\limits{x\in U} \fF(x)\), where \(\fF\) is a given 2-functor.

      \begin{Th} \label{thm:global-sections-2limit}
        The 2-functor (pseudofunctor) \(x \mapsto \coex(\cO_{X,x}, \cA_x)\) is a stack
        of symmetric categorical groups over \(X\) whose global section is canonically
        equivalent to \(\coex(X, \cA)\).
      \end{Th}

      \begin{proof}
        We have shown the 2-functoriality of \(x \mapsto \coex(\cO_{X,x}, \cA_x)\) in
        Lemma~\ref{lem:rho_2_functor}. As discussed just above, this already states
        that it extends to a unique stack over the underlying topological space of
        \(X\). At a given open subset \(U\subseteq X\), the value of the stack is thus
        the 2-limit \(\tl_{x\in U}\coex(\cO_{X,x}, \cA_x)\).

        It follows thus that the only thing we have to show is the equivalence
        \[ \coex(X, \cA) \;\simeq\; \tl_{x \in X}\, \coex(\cO_{X,x}, \cA_x). \]
        To do so, we construct mutually inverse equivalences.

        Given \(\coex(X, \cA)\ni \cE = 0 \to \cA \to \cG \to \cO_X \to 0\),
        Lemma~\ref{lem:localisation_coext} gives us\(\cE_x := \cG_x \in
        \coex(\cO_{X,x}, \cA_x)\) for each \(x\) and Lemma~\ref{lem:rho_2_functor} the
        morphisms induced by \(y \leq x\). That the cocycle condition holds is
        essentially just the sheaf condition. Hence, we land in each entry of the
        2-limit and thus, in the 2-limit itself.

        For the reverse, let \((\cE_x, \phi_{x,y})\) be a compatible collection in the
        2-limit, with \(\cG_x\) being the extension schemes. Since these are schemes,
        they glue and its value at the global to the monoid scheme \(\cG\). (An other
        way of saying that is: Define a monoid scheme whose value at a given stalk
        \(x\in X\) is \(\cG_x\).) The morphism \(\cG_x\xto{\pi_x} \cF_x\) glues as well
        to a morphism of schemes (its pointwise defined) and likewise, we get the
        system of abelian groups \(\cA\) (this too, is pointwise defined). That cocycle
        condition will hold as it is just the 2-limit condition.

        It is clear that these are mutually inverse and clearly, its natural
        (functorial).
      \end{proof}

      This allows us to see \(\coex(X, \cA)\) as a stack. While we have only shown this
      for the global section, as we made no assumptions of the monoid scheme \(X\), by
      restricting ourselves to a given open subset, we can simply use the above theorem
      to say that indeed, \(\coex(X, \cA)\) is a stack over \(X\).

\section{The Global Grillet Complex and Descent for Sheaf Coextensions}
  \label{sec:classification_sheaves}

  The aim of this section is to develop a global version of the Grillet complex. We
  will start with the case of monoid sheaves and move on to the more important case
  (for us) of monoid schemes. In the latter, we will touch the on stack structure of
  \(\coex(X, \cA)\) and see its implication in the Grillet complex case.

\subsection{The global Grillet complex and classification for monoid sheaves}

\subsubsection{The bicomplex}

      Let \(\cF\) be a sheaf of monoids over a finite poset \(P\) and \(\cA \in
      \cH(\cF)\) a system of abelian groups. Recall from Definition~\ref{def:grillet}
      that there is a truncated Grillet complex
      \[
        \sC^0(\cF_p, \cA_p) \xto{\partial^0} \sC^1(\cF_p, \cA_p) \xto{\partial^1}
        \sC^2(\cF_p, \cA_p)
      \]
      for each \(p \in P\). These arrange into the truncated bicomplex
      \(C_P^{\bullet,*}(\cF, \cA)\), with \(0\leq * \leq 2\), given by
      \begin{align*}
        C_P^{0,*}(\cF, \cA) & \;=\; \prod_{p \in P} \sC^*(\cF_p, \cA_p),                    \\
        C_P^{1,*}(\cF, \cA) & \;=\; \prod_{q \leq p} \sC^*(\cF_p, \phi_{p,q}^*\cA_q),       \\
        C_P^{2,*}(\cF, \cA) & \;=\; \prod_{r \leq q \leq p} \sC^*(\cF_p, \phi_{p,r}^*\cA_r) \\
                            & \cdots
      \end{align*}
      The vertical coboundary maps are the Grillet maps \(\partial^*\) defined
      pointwise (meaning, on each stalk, and each factor of the products). The
      horizontal coboundary maps are the standard maps associated to the cochain
      complex associated to a poset \(P\). We are only interested in a small part of
      this bicomplexes, namely the map
      \[ \delta: C_P^{0,0}(\cF, \cA) \;\rightarrow\; Z^1(\cF, \cA). \]
      Here, \(Z^1(\cF, \cA)\) is the sheaf version of the kernel \(\ker\partial^1\) in
      Theorem~\ref{thm:classification}, which is obviously our main target to
      generalise in the non-affine setting.

\subsubsection{The groups \(C_P^{0,0}\) and \(Z^1\)}

      The two terms that are going to play the major role in our discussion of the
      Grillet complex in the non-affine case are \(C_P^{0,0}\) and \(Z^1\).
      Subsequently, we will give more explicit descriptions of these, as well as the
      maps between the two.

      For \(C_P^{0,0}(\cF, \cA) = \prod_{p \in P} \sC^0(\cF_p, \cA_p)\), an element \(h
      \in C_P^{0,0}\) is a function that assigns to each \(p \in P\) and \(x \in
      \cF_p\) an element \(h(p,x) \in \cA_p(x)\).

      To give \(Z^1\), we have write things out a bit more in detail:

      \begin{De}
        The abelian group \(Z^1(\cF, \cA)\) consists of pairs \((f, g)\). The first
        entry \(f\) assigns to each \(p \in P\) and \(x, y \in \cF_p\) an element
        \(f(p,x,y) \in \cA_p(xy)\). The second coordinate \(g\) assigns to each pair
        \(q \leq p\) in \(P\) and \(x \in \cF_p\) an element \(g(p,q,x) \in
        \cA_q(\phi_{p,q}(x))\). These are subject to the following four conditions:
        \begin{eqnarray}
          0 & = & x_* f(p,y,z) - f(p,xy,z) + f(p,x,yz) - z_* f(p,x,y), \label{eq:Z1-i}                  \\
          0 & = & f(p,x,y) - f(p,y,x), \label{eq:Z1-ii}                                                 \\
          0 & = & g(p,r,x) - \alpha_{q,r}\bigl(g(p,q,x)\bigr) - g(q,r,\phi_{p,q}(x)), \label{eq:Z1-iii} \\
          0 & = & g(p,q,xy) - y_* g(p,q,x) - x_* g(p,q,y) - f\bigl(q,\phi_{p,q}(x),\phi_{p,q}(y)\bigr) + \alpha_{p,q}(f(p,x,y)), \label{eq:Z1-iv}
        \end{eqnarray}
        where \(x,y,z \in \cF_p, r \leq q \leq p\).
      \end{De}

      \begin{Rem}
        Conditions \eqref{eq:Z1-i} and \eqref{eq:Z1-ii} together say that \(f(p,-,-)
        \in \ker\partial^1(\cF_p, \cA_p)\) for each \(p\). This makes \(f\) a class of
        Grillet 1-cocycles in the affine sense. Condition \eqref{eq:Z1-iii} is the
        cocycle condition. Condition \eqref{eq:Z1-iv} is the compatibility condition.
        It says that the failure of \(g(p,q,-)\) to be a derivation is exactly
        controlled by the difference between \(f(q,-,-)\) and \(\alpha_{p,q}
        f(p,-,-)\).
      \end{Rem}

      The coboundary map \(\delta: C_P^{0,0} \to Z^1(\cF, \cA)\) is defined by
      \(\delta(h) \;=\; (f_h,\, g_h)\), where
      \begin{eqnarray}
        f_h(p,x,y) & = & x_* h(p,y) - h(p,xy) + y_* h(p,x) \;=\; (\partial^0 h_p)(x,y) \label{eq:delta} \\
        g_h(p,q,x) & = & \alpha_{p,q}(h(p,x)) - h(q, \phi_{p,q}(x)) \label{eq:delta-g}
      \end{eqnarray}

      We now check that makes sense:

      \begin{Le}
        For every \(h \in C_P^{0,0}\), we have \(\delta(h) = (f_h, g_h)\in Z^1(\cF,
        \cA)\).
      \end{Le}

      \begin{proof}
        Conditions \eqref{eq:Z1-i} and \eqref{eq:Z1-ii} hold because \(f_h(p,-,-) =
        \partial^0 h_p\) lies in \(\ker\partial^1(\cF_p, \cA_p)\) and is symmetric. For
        \eqref{eq:Z1-iii}, we have
        \begin{align*}
          \alpha_{q,r}(g_h(p,q,x)) + g_h(q,r,\phi_{p,q}(x))
                                                            & = \alpha_{q,r}(\alpha_{p,q}(h(p,x)) - h(q,\phi_{p,q}(x)))                  \\
                                                            & \quad + \alpha_{q,r}(h(q,\phi_{p,q}(x))) - h(r, \phi_{q,r}(\phi_{p,q}(x))) \\
                                                            & = \alpha_{q,r}(\alpha_{p,q}(h(p,x))) - h(r, \phi_{p,r}(x))                 \\
                                                            & = \alpha_{p,r}(h(p,x)) - h(r, \phi_{p,r}(x))                               \\
                                                            & = g_h(p,r,x),
        \end{align*}
        using the cocycle condition \(\alpha_{p,r} = \phi_{p,q}^*(\alpha_{q,r}) \circ
        \alpha_{p,q}\) and \(\phi_{q,r} \circ \phi_{p,q} = \phi_{p,r}\). For
        \eqref{eq:Z1-iv}, a direct computation using \eqref{eq:delta} and
        \eqref{eq:delta-g} gives that both \(g(p,q,xy) - y_* g(p,q,x) - x_* g(p,q,y)\)
        and \(f\bigl(q,\phi_{p,q}(x),\phi_{p,q}(y)\bigr) + \alpha_{p,q}(f(p,x,y))\)
        equal
        \[
          \alpha_{p,q}(x_* h(p,y) + y_* h(p,x) - h(p,xy)) -
          ((\phi_{p,q}(x))_* h(q,\phi_{p,q}(y)) +
          (\phi_{p,q}(y))_* h(q,\phi_{p,q}(x)) - h(q,\phi_{p,q}(xy))).
        \]
      \end{proof}

      We are now in a position to define the cohomologies
      \[
        \sD^0_P(\cF, \cA) \;:=\; \ker\delta, \qquad \sD^1_P(\cF, \cA) \;:=\;
        \CoKer\delta \;=\; Z^1(\cF,\cA)\,/\,\im\delta.
      \]

\subsubsection{Classification}

      \begin{Th}
        Let \(\cF\) be a sheaf of monoids over a finite poset \(P\) and \(\cA \in
        \cH(\cF)\). There is an equivalence of symmetric categorical groups
        \[
          \coex(\cF, \cA) \;\simeq\; \bigl[C_P^{0,0}(\cF, \cA) \xto{\;\delta\;}
          Z^1(\cF, \cA)\bigr],
        \]
        where the right-hand side is the groupoid associated to the morphism of abelian
        groups \(\delta\). In particular,
        \[ \pi_0(\coex(\cF,\cA)) \cong \sD^1_P(\cF,\cA). \]
      \end{Th}

      Much like in the affine case, we prove this directly by constructing mutually
      inverse functors.

      \begin{proof}
        \textbf{Step 1: \(Z^1(\cF,\cA) \to \coex(\cF,\cA)\):} Let \((f,g) \in
        Z^1(\cF,\cA)\). Our aim is to construct a coextension \(0 \to \cA \to \cG \to
        \cF \to 0\). Recall that we have to do this pointwise. Let \(p\in P\). Define
        \[ \cG_p \;:=\; \{(a, x) \mid x \in \cF_p,\; a \in \cA_p(x)\} \]
        and endow it with multiplication by setting
        \begin{equation}
          (a,x)(b,y) \;:=\; \bigl(x_*(b) + y_*(a) + f(p,x,y),\; xy\bigr).
        \end{equation}
        We claim that this is a commutative monoid. Indeed, as \(f\)
        (Condition~\eqref{eq:Z1-ii}) is symmetric, \(\cG_p\) is commutative. Next, we
        observe that associativity is equivalent to \(\partial^1 f(p,-,-,-) = 0\),
        which is just Condition~\eqref{eq:Z1-i}. Finally, the identity element is
        played by \((-f(p,1,1), 1_{\cF_p})\), as \(f(p,1,x) = x_*(f(p,1,1))\). This is
        Condition~\eqref{eq:Z1-i} where we set \(x = y = 1\), and the exact same type
        of argument we used in the affine case (Theorem~\ref{thm:classification}). This
        makes \(\cG_p\) a coextension of \(\cF_p\) by \(\cA_p\), with cocycle
        \(f(p,-,-)\).

        For \(q \leq p\), we now have to define the restriction map \(\phi_{p,q}^{\cG}:
        \cG_p \to \cG_q\). For this, we send \((x,y)\) as follows:
        \begin{equation}
          \phi_{p,q}^{\cG}(a, x) \;:=\; \bigl(\alpha_{p,q}(a) + g(p,q,x),\;
          \phi_{p,q}^{\cF}(x)\bigr).
        \end{equation}
        This is a monoid homomorphism as we have
        \begin{align*}
          \phi_{p,q}^{\cG}\bigl((a,x)(b,y)\bigr)
                                                 & = \phi_{p,q}^{\cG}\bigl(x_*(b)+y_*(a)+f(p,x,y),\; xy\bigr) \\
                                                 & = \bigl(\alpha_{p,q}(x_*(b)+y_*(a)+f(p,x,y)) + g(p,q,xy),\;
          \phi_{p,q}(xy)\bigr)
        \end{align*}
        on the one hand and on the other
        \begin{align*}
          \phi_{p,q}^{\cG}(a,x) \cdot \phi_{p,q}^{\cG}(b,y) & = \bigl(\alpha_{p,q}(a)+g(p,q,x),\; \phi_{p,q}(x)\bigr) \cdot \bigl(\alpha_{p,q}(b)+g(p,q,y),\; \phi_{p,q}(y)\bigr) \\
                                                            & = \bigl(\phi_{p,q}(x)_*(\alpha_{p,q}(b)+g(p,q,y)) + \phi_{p,q}(y)_*(\alpha_{p,q}(a)+g(p,q,x))                       \\
                                                            & \qquad + f(q,\phi_{p,q}(x),\phi_{p,q}(y)),\; \phi_{p,q}(x)\phi_{p,q}(y)\bigr).
        \end{align*}
        We need these to agree, which means that both the first and second coordinates
        have to agree. For the first components to agree, we need
        \[
          \alpha_{p,q}(x_*(b))
          + \alpha_{p,q}(y_*(a))
          + \alpha_{p,q}(f(p,x,y))
          + g(p,q,xy)
        \]
        to equal
        \[
          \phi_{p,q}(x)_*\alpha_{p,q}(b)
          + \phi_{p,q}(y)_*\alpha_{p,q}(a)
          + f(q,\phi_{p,q}(x),\phi_{p,q}(y))
          + \phi_{p,q}(y)_*g(p,q,x)
          + \phi_{p,q}(x)_*g(p,q,y).
        \]
        Using the fact that \(\alpha_{p,q}: \cA_p \to \phi_{p,q}^*\cA_q\) is a morphism
        of abelian group systems (natural transformations), we see that
        \begin{eqnarray*}
          \alpha_{p,q}(x_*(b)) & = & \phi_{p,q}(x)_*\alpha_{p,q}(b) \\
          \alpha_{p,q}(x_*(b)) & = & \phi_{p,q}(y)_*\alpha_{p,q}(a) \\
        \end{eqnarray*}
        The remaining terms, meaning
        \[
          \alpha_{p,q}(f(p,x,y)) + g(p,q,xyl) \;=\; \phi_{p,q}(y)_*g(p,q,x) +
          \phi_{p,q}(x)_*g(p,q,y), + f(q,\phi_{p,q}(x),\phi_{p,q}(y))
        \]
        reduce to Condition~\eqref{eq:Z1-iv}.

        For the second one, we need to have
        \[ \phi_{p,q}(xy) \;=\; \phi_{p,q}(x)\phi_{p,q}(y), \]
        which basically means functoriality. It follows from
        Condition~\eqref{eq:Z1-iii} in the following way:

        We have
        \begin{eqnarray*}
          \phi_{q,r}^{\cG}(\phi_{p,q}^{\cG}(a,x))) & = & \bigl(\alpha_{q,r}(\alpha_{p,q}(a) + g(p,q,x)) + g(q,r,\phi_{p,q}(x)),\; \phi_{p,r}(x)\bigr) \\
                                                   & = & (\alpha_{p,r}(a) + g(p,r,x),\; \phi_{p,r}(x))                                                \\
                                                   & = & \phi_{p,r}^{\cG}(a,x)
        \end{eqnarray*}
        by Condition~\eqref{eq:Z1-iii} and the cocycle condition \(\alpha_{q,r} \circ
        \alpha_{p,q} = \alpha_{p,r}\).

        Hence, \(\cG\) is a functor from the poset \(P\) and thus, a well-defined sheaf
        monoids over its poset topology. We define the morphism of sheaves \(\pi: \cG
        \to \cF\) by \(\pi_p(a,x) = x\). This is a coextension with \(\cA_p\)-action on
        \(\pi_p^{-1}(x)\) given by \(z \pt (a,x) = (z+a, x)\; z \in \cA_p(x)\). To see
        that it satisfies the regularity and compatibility conditions we can use the
        same type of argument as in the affine case (Theorem~\ref{thm:classification}),
        pointwise.

        \textbf{Step 2: \(\coex(\cF,\cA)\to Z^1(\cF,\cA)\).} Let \(0 \to \cA \to \cG
        \xto{\pi} \cF \to 0\) be a coextension of \(\cF\) by \(\cA\). Since \(\pi\) is
        a surjective sheaf homomorphism over a poset topology, it is surjective at each
        point, and thus, we can choose sections \(\tau_p: \cF_p \to \cG_p\) of
        \(\pi_p\) for each \(p \in P\) (these are just set-maps). By the affine
        classification (Theorem~\ref{thm:classification}, Step~2) applied pointwise, we
        can define \( f(p,x,y) \in \cA_p(xy) \) by
        \[ \tau_p(x)\tau_p(y) \;=\; f(p,x,y) \pt \tau_p(xy), \]
        so that \(f(p,-,-) \in \ker\partial^1(\cF_p, \cA_p)\),
        Conditions~\eqref{eq:Z1-i} and \eqref{eq:Z1-ii} will hold by
        Theorem~\ref{thm:classification}. In like manner, we can define \( g(p,q,x)
        \;\in\; \cA_q(\phi_{p,q}(x)) \) through
        \[ \phi_{p,q}^{\cG}(\tau_p(x)) \;=\; g(p,q,x) \pt \tau_q(\phi_{p,q}(x)). \]
        Condition~\eqref{eq:Z1-iii} is just functoriality \(\phi_{q,r}^{\cG} \circ
        \phi_{p,q}^{\cG} = \phi_{p,r}^{\cG}\) applied to \(\tau_p(x)\). For
        Condition~\eqref{eq:Z1-iv}, we apply \(\phi_{p,q}^{\cG}\) to
        \(\tau_p(x)\tau_p(y) = f(p,x,y) \pt \tau_p(xy)\) and use the definition of
        \(g\) on each factor.

        \textbf{Step 3:} We now want to show that what we did is well-defined (does not
        depend on the choice of section) and that the two functors are mutually
        inverse. Let \(\tau'\) be a different choice of section (meaning, collection of
        point-wise sections). By regularity, we can assemble a "function" \(h_p:
        \cF_p\to \cA_p\) such that \(\tau_p' = h_p \pt \tau_p\) will hold. This changes
        \((f,g)\) by
        \[ (f', g') \;=\; (f + f_{h}, g + g_{h}) \;=\; (f,g) + \delta(h). \]
        Here \(h = (h_p)_{p \in P} \in C_P^{0,0}\). Hence, two sections will give the
        same class in \(Z^1/\im\delta = \sD^1_P(\cF,\cA)\) and thus, the class does not
        depend on the choice of the section.

        It is not hard to see that the two constructions are mutually inverse (up to
        isomorphism, of course, since these are 2-categorical objects). Thus, we have
        shown the equivalence of groupoids \(\coex(\cF,\cA) \simeq [C_P^{0,0}
        \xto{\delta} Z^1]\), as desired. It is also not hard to see that the structures
        of symmetric categorical groups are preserved by this equivalence.
      \end{proof}

      \begin{Rem}
        When \(P\) is a single point, we will refer to it as the affine case, and it
        indeed agrees exactly with Theorem~\ref{thm:classification}. Note that, this
        does not mean affine in the scheme theoretical sense, where the existence of a
        unique maximal element (rather than \(P\) being solely a single element) is
        equivalent to affineness. For this, however, we would need a type of
        quasi-coherence condition, meaning for monoids sheaves, we have to consider
        monoid schemes, see the following section.
      \end{Rem}

\section{The Triviality in the Semilattice Case} \label{sec:semilattice}

  As already mentioned in Section~\ref{sec:spec}, semilattices play an important role
  in monoid theory. A somewhat natural question is to consider what happens in this
  setting. As extension theory is closely tied to groups, and semilattice theory are in
  a way orthogonal to group theory, is perhaps not completely unexpected that extension
  theory of semilattices is trivial. We aim to show exactly that in this section.

  Throughout this section, \(M\) will denote a semilattice. The category \(\bH(M)\) is
  quite non-trivial in this case. For instance, let \(M = I = \{e, t\}\) be the field
  with two elements under multiplication. In this case, any pair \((\phi, p)\), where
  \(\phi: A \to B\) is a group homomorphism and \(p: B \to B\) an idempotent
  endomorphism with \(\phi=p\phi\), defines a system by setting \(\cA(e) = A\),
  \(\cA(t) = B\), \(t_*|_{\cA(e)} = \phi\), and \(t_*|_{\cA(t)} = p\).

  Despite this richness, the coextension theory over a semilattice is completely
  trivial.

\subsection{The affine case}

    \begin{Th} \label{thm:sl_trivial}
      Let \(M\) be a semilattice and \(\cA \in \cH(M)\). Then \(\sD^0(M, \cA) = 0\) and
      \(\sD^1(M, \cA) = 0\). In particular, every coextension \(0 \to \cA \to N
      \xto{\pi} M \to 0\) splits uniquely.
    \end{Th}

    \begin{proof}
      A coextension is trivial if and only if \(\partial^0:{\sf C}^0(M,\cA)\to
      \ker\partial ^1\) is an isomorphism, and this is exactly what we aim to prove.

      Now we define \(\zeta:\ker\partial \to {\sf C}^0(M,\cA)\) by
      \[ \zeta(f)(a) \;=\; \phi_a(f(a,a))=2a_*f(a,a)-f(a,a) \]
      and claim that it is the inverse of \(\partial^0\). Here \(f\in \ker
      \partial^1\). Recall this means that \(f\) is a function which assigns to any
      pair of elements \((a,b)\in M\) an element \(f(a,b)\in \cA(ab)\), such that
      \(f(a,b) = f(b,a)\) and the following cocycle condition holds
      \[ a_*f(b,c) - f(ab,c) + f(a,bc) - c_*f(a,b) \;=\; 0. \]

      Observe that for any \(a\in M\), the relation \(a^2 = a\) implies that
      \(a_*:\cA(a)\to\cA(a)\) is an idempotent endomorphism of \(\cA(a)\). We set
      \[ \phi_a:=2a_*-\id_{\cA(A)}:\cA(a)\to \cA(a). \]
      It is an involutive automorphism. In fact
      \[ \phi_a^2 \;=\; (2a - \id)^2 \;=\; 4a^2 - 4a + \id \;=\; \id, \]
      Take \(g\in {\sf C}^0(,M\cA)\). We have
      \[
        \left (\zeta\partial ^0(g)\right)(a)
        \;=\; \phi_a\left(\partial ^0(g)(a,a)\right)
        \;=\; \phi_a(2a_*g(a)-g(a))
        \;=\; \phi_a^2(g(a))=g(a)
      \]
      and thus \(\zeta\circ \partial ^0=\id_{{\sf C}^0(,M\cA)}\).

      It remains to check that \( \partial ^0\zeta (f)=f\) for any \(f\in \ker \partial
      ^1\). This requires some manipulation of the cocycle condition. Consider the
      following three equalities, which are obtained by applying the cocycle condition
      to the triples \((a,a, b)\), \((b,b,a)\) and \((a,b,ab)\), respectively.
      \begin{eqnarray}
        f(a,ab)  & = & b_* f(a,a) - (a_* - \id)f(a,b) \label{eq:a-ab} \\
        f(b,ab)  & = & a_* f(b,b) - (b_* - \id)f(a,b) \label{eq:b-ab} \\
        f(ab,ab) & = & a_* f(b,ab) + f(a,ab) - (ab)_*f(a,b) \label{eq:ab-ab}.
      \end{eqnarray}
      Substitute Identities~\eqref{eq:a-ab} and \eqref{eq:b-ab} in \eqref{eq:ab-ab}, to
      obtain
      \[ f(ab,ab) \;=\; a_* f(b,b) + b_* f(a,a) + f(a,b) - 2(ab)_*f(a,b). \]
      Using this, we get
      \[
        \begin{aligned}
          \zeta(f)(ab) & \;=\; 2a_*b_*f(ab,ab) - f(ab,ab)                                    \\
                       & \;=\; 2a_*b_*f(b,b) + 2a_*b_*f(a,a) + 2a_*b_*f(a,b) - 4a_*b_*f(a,b) \\
                       & \quad - a_* f(b,b) - b_* f(a,a) - f(a,b) + 2(ab)_*f(a,b)            \\
                       & \;=\; a_*\zeta(f)(b)+b_*\zeta(f)(a)-f(a,b).
        \end{aligned}
      \]
      It follows that
      \[ \partial ^0\zeta (f)(a,b) \;=\; a_*\zeta(b)-\zeta(ab)+b_*\zeta(a)=f(a,b). \]
      as desired.
    \end{proof}

\subsection{The global case}

    Having shown that semilattices only have trivial coextensions, we now aim to show
    that the gluings of semilattices likewise have only trivial coextensions. In other
    words, we consider a monoid scheme \(X\) whose stalk \(\cO_{X,x}\) at every \(x\in
    X\) is a semilattice. We call such a monoid scheme a \emph{semilattice scheme}.
    Indeed, we will work in a more general setting, namely general semilattice sheaves
    over a poset \(P\).

    The following result extends Theorem~\ref{thm:sl_trivial} to semilattice sheaves.

    \begin{Th}
      Let \(\cF\) be a sheaf of semilattices. Then \(\sD^n(\cF, \cA) = 0\) for \(n = 0,
      1\). In particular, every coextension \(0 \to \cA \to \cG \to \cF \to 0\) has a
      unique splitting.
    \end{Th}

    \begin{proof}
      The corresponding global Grillet's complex is constructed as the truncation of
      the total complex of a bicomplex. By Theorem~\ref{thm:sl_trivial}, each column
      has trivial homology and it thus follows that the same is true for the total
      complex.
    \end{proof}

\end{document}